\documentclass[11pt]{amsart}

\usepackage{amsmath}
\usepackage{amscd}
\usepackage{amssymb}
\usepackage{boxedminipage}
\usepackage{enumerate}
\usepackage{graphics}
\usepackage{epic}
\usepackage{curves}
\usepackage{footmisc}

\begin{document}

\theoremstyle{plain} \newtheorem{theorem}{Theorem}[section]
\theoremstyle{plain} \newtheorem{lemma}[theorem]{Lemma}
\theoremstyle{plain} \newtheorem{proposition}[theorem]{Proposition}
\newtheorem{axioms}[theorem]{Axioms}
\newtheorem{corollary}[theorem]{Corollary}
\newtheorem{problem}{Problem}
\newtheorem{subproblem}[problem]{Subproblem}
\newtheorem{conjecture}[theorem]{Conjecture}
\newtheorem{conjecture*}[]{Conjecture}
\newtheorem{matheorem}[theorem]{Main Theorem}
\newtheorem{claim}[problem]{Claim}

\newcommand{\nr}{\refstepcounter{theorem}  
                   \noindent {\thetheorem .}}
\newcommand{\defi}{\medskip \noindent {\it Definition \nr} }
\newcommand{\defifin}{\medskip}
\newcommand{\eks}{\medskip \noindent {\it Example \nr} }
\newcommand{\eksfin}{\medskip}
\newcommand{\rem}{\medskip \noindent {\it Remark \nr} }
\newcommand{\remfin}{\medskip}
\newcommand{\obs}{\medskip \noindent {\it Observation \nr} }
\newcommand{\obsfin}{\medskip}
\newcommand{\note}{\medskip \noindent {\it Notation \nr} }
\newcommand{\notefin}{\medskip}

\newcommand{\llabel}{\addtocounter{theorem}{-1}
\refstepcounter{theorem} \label}

\newcommand{\psp}[1]{{{\bf P}^{#1}}}
\newcommand{\psr}[1]{{\bf P}(#1)}
\newcommand{\op}{{\mathcal O}}
\newcommand{\opw}{\op_{\psr{W}}}
\newcommand{\go}{\op}

%Initial ideals
\newcommand{\ini}[1]{\text{in}(#1)}
\newcommand{\gin}[1]{\text{gin}(#1)}
\newcommand{\kr}{{\Bbbk}}
\newcommand{\pd}{\partial}
\renewcommand{\tt}{{\bf t}}

%Kategorier

\newcommand{\coh}{{{\text{{\rm coh}}}}}

%Modulkategorier

\newcommand{\modv}[1]{{#1}\text{-{mod}}}
\newcommand{\modstab}[1]{{#1}-\underline{\text{mod}}}

\newcommand{\sut}{{}^{\tau}}
\newcommand{\sumit}{{}^{-\tau}}
\newcommand{\til}{\thicksim}

\newcommand{\totp}{\text{Tot}^{\prod}}
\newcommand{\dsum}{\bigoplus}
\newcommand{\dprod}{\prod}
\newcommand{\lsum}{\oplus}
\newcommand{\lprod}{\Pi}

% Algebraer
\newcommand{\La}{{\Lambda}}

\newcommand{\sirstj}{\circledast}

% Knipper
\newcommand{\she}{\EuScript{S}\text{h}}
\newcommand{\cm}{\EuScript{CM}}
\newcommand{\cmd}{\EuScript{CM}^\dagger}
\newcommand{\cmri}{\EuScript{CM}^\circ}
\newcommand{\cler}{\EuScript{CL}}
\newcommand{\clerd}{\EuScript{CL}^\dagger}
\newcommand{\clerri}{\EuScript{CL}^\circ}
\newcommand{\gor}{\EuScript{G}}
\newcommand{\gF}{\mathcal{F}}
\newcommand{\gG}{\mathcal{G}}
\newcommand{\gM}{\mathcal{M}}
\newcommand{\gE}{\mathcal{E}}
\newcommand{\gD}{\mathcal{D}}
\newcommand{\gI}{\mathcal{I}}
\newcommand{\gP}{\mathcal{P}}
\newcommand{\gK}{\mathcal{K}}
\newcommand{\gL}{\mathcal{L}}
\newcommand{\gS}{\mathcal{S}}
\newcommand{\gC}{\mathcal{C}}
\newcommand{\gO}{\mathcal{O}}
\newcommand{\gJ}{\mathcal{J}}

\newcommand{\dlim} {\varinjlim}
\newcommand{\ilim} {\varprojlim}

%Kategorier
\newcommand{\CM}{\text{CM}}
\newcommand{\Mon}{\text{Mon}}

%Kategorieer av komplekser

\newcommand{\Kom}{\text{Kom}}

% Begreper homologisk alebra

\newcommand{\EH}{{\mathbf H}}
\newcommand{\res}{\text{res}}
\newcommand{\Hom}{\text{Hom}}
\newcommand{\inhom}{{\underline{\text{Hom}}}}
\newcommand{\Ext}{\text{Ext}}
\newcommand{\Tor}{\text{Tor}}
\newcommand{\ghom}{\mathcal{H}om}
\newcommand{\gext}{\mathcal{E}xt}
\newcommand{\id}{\text{{id}}}
\newcommand{\im}{\text{im}\,}
\newcommand{\codim} {\text{codim}\,}
\newcommand{\resol}{\text{resol}\,}
\newcommand{\rank}{\text{rank}\,}
\newcommand{\lpd}{\text{lpd}\,}
\newcommand{\coker}{\text{coker}\,}
\newcommand{\supp}{\text{supp}\,}

%Avbildninger og andre symbolforkortelser

\newcommand{\sus}{\subseteq}
\newcommand{\sups}{\supseteq}
\newcommand{\pil}{\rightarrow}
\newcommand{\vpil}{\leftarrow}
\newcommand{\rpil}{\leftarrow}
\newcommand{\lpil}{\longrightarrow}
\newcommand{\inpil}{\hookrightarrow}
\newcommand{\pils}{\twoheadrightarrow}
\newcommand{\projpil}{\dashrightarrow}
\newcommand{\dotpil}{\dashrightarrow}
\newcommand{\adj}[2]{\overset{#1}{\underset{#2}{\rightleftarrows}}}
\newcommand{\mto}[1]{\stackrel{#1}\longrightarrow}
\newcommand{\vmto}[1]{\stackrel{#1}\longleftarrow}
\newcommand{\eqv}{\Leftrightarrow}
\newcommand{\impl}{\Rightarrow}

\newcommand{\iso}{\cong}
\newcommand{\te}{\otimes}
\newcommand{\into}[1]{\hookrightarrow{#1}}
\newcommand{\ekv}{\Leftrightarrow}
\newcommand{\equi}{\simeq}
\newcommand{\isopil}{\overset{\cong}{\lpil}}
\newcommand{\equipil}{\overset{\equi}{\lpil}}
\newcommand{\ispil}{\isopil}
\newcommand{\vvi}{\langle}
\newcommand{\hvi}{\rangle}
\newcommand{\susneq}{\subsetneq}

%Notasjonsforkortelser

\newcommand{\xd}{\check{x}}
\newcommand{\ortog}{\bot}
\newcommand{\tL}{\tilde{L}}
\newcommand{\tM}{\tilde{M}}
\newcommand{\tH}{\tilde{H}}
\newcommand{\tvH}{\widetilde{H}}
\newcommand{\tvh}{\widetilde{h}}
\newcommand{\tV}{\tilde{V}}
\newcommand{\tS}{\tilde{S}}
\newcommand{\tT}{\tilde{T}}
\newcommand{\tR}{\tilde{R}}
\newcommand{\tf}{\tilde{f}}
\newcommand{\ts}{\tilde{s}}
\newcommand{\tp}{\tilde{p}}
\newcommand{\tr}{\tilde{r}}
\newcommand{\tfst}{\tilde{f}_*}
\newcommand{\empt}{\emptyset}
\newcommand{\bfa}{{\bf a}}
\newcommand{\la}{\lambda}

\newcommand{\ome}{\omega_E}

\newcommand{\bevis}{{\bf Proof. }}
\newcommand{\demofin}{\qed \vskip 3.5mm}
\newcommand{\nyp}[1]{\noindent {\bf (#1)}}
\newcommand{\demo}{{\it Proof. }}
\newcommand{\demodone}{\demofin}
\newcommand{\parg}{{\vskip 2mm \addtocounter{theorem}{1}  
                   \noindent {\bf \thetheorem .} \hskip 1.5mm }}

\newcommand{\red}{{\text{red}}}
\newcommand{\lcm}{{\text{lcm}}}

% Simplisielle komplekser

\newcommand{\dl}{\Delta}
\newcommand{\cdel}{{C\Delta}}
\newcommand{\cdelp}{{C\Delta^{\prime}}}
\newcommand{\dlst}{\Delta^*}
\newcommand{\Sdl}{{\mathcal S}_{\dl}}
\newcommand{\lk}{\text{lk}}
\newcommand{\lkd}{\lk_\Delta}
\newcommand{\lkp}[2]{\lk_{#1} {#2}}
\newcommand{\del}{\Delta}
\newcommand{\delr}{\Delta_{-R}}
\newcommand{\dd}{{\dim \del}}

\renewcommand{\aa}{{\bf a}}
\newcommand{\bb}{{\bf b}}
\newcommand{\cc}{{\bf c}}
\newcommand{\xx}{{\bf x}}
\newcommand{\yy}{{\bf y}}
\newcommand{\zz}{{\bf z}}

\newcommand{\pnm}{{\bf P}^{n-1}}
\newcommand{\opnm}{{\go_{\pnm}}}
\newcommand{\ompnm}{\omega_{\pnm}}

\newcommand{\pn}{{\bf P}^n}
\newcommand{\hele}{{\bf Z}}

\newcommand{\dt}{{\displaystyle \cdot}}
\newcommand{\st}{\hskip 0.5mm {}^{\rule{0.4pt}{1.5mm}}}              
\newcommand{\disk}{\scriptscriptstyle{\bullet}}

\def\CC{{\mathbb C}}
\def\GG{{\mathbb G}}
\def\ZZ{{\mathbb Z}}
\def\NN{{\mathbb N}}
\def\OO{{\mathbb O}}
\def\QQ{{\mathbb Q}}
\def\VV{{\mathbb V}}
\def\PP{{\mathbb P}}
\def\EE{{\mathbb E}}
\def\FF{{\mathbb F}}
\def\AA{{\mathbb A}}

\newcommand{\GRID}{\vskip 1cm 
\begin{picture}(400,120)
\put(0,120){\line(0,-1){120}}
\put(20,120){\line(0,-1){120}}
\put(40,120){\line(0,-1){120}}
\put(60,120){\line(0,-1){120}}
\put(80,120){\line(0,-1){120}}
\put(100,120){\line(0,-1){120}}
\put(120,120){\line(0,-1){120}}
\put(140,120){\line(0,-1){120}}
\put(160,120){\line(0,-1){120}}
\put(180,120){\line(0,-1){120}}
\put(200,120){\line(0,-1){120}}
\put(220,120){\line(0,-1){120}}
\put(240,120){\line(0,-1){120}}
\put(260,120){\line(0,-1){120}}
\put(280,120){\line(0,-1){120}}
\put(300,120){\line(0,-1){120}}
\put(320,120){\line(0,-1){120}}
\put(340,120){\line(0,-1){120}}
\put(360,120){\line(0,-1){120}}
\put(380,120){\line(0,-1){120}}
\put(400,120){\line(0,-1){120}}

\put(0,0){\line(1,0){400}}
\put(0,20){\line(1,0){400}}
\put(0,40){\line(1,0){400}}
\put(0,60){\line(1,0){400}}
\put(0,80){\line(1,0){400}}
\put(0,100){\line(1,0){400}}
\put(0,120){\line(1,0){400}}
\end{picture}
}

\newcommand{\trekantA}{\vskip 1cm
\begin{picture}(300,140)
\put(25,30){$y^2$}
\put(40,40){\circle{1}}
\put(40,40){\line(1,2){40}}

\put(75,124){$x^2$}
\put(80,120){\circle{1}}
\put(80,120){\line(1,-2){40}}

\put(122,30){$z^2$}
\put(120,40){\circle{1}}
\put(120,40){\line(-1,0){80}}

\put(46,80){$xy$}
\put(60,80){\circle{1}}
\put(60,80){\line(1,0){40}}

\put(105,80){$xz$}
\put(100,80){\circle{1}}
\put(100,80){\line(-1,-2){20}}

\put(80,30){$yz$}
\put(80,40){\circle{1}}
\put(80,40){\line(-1,2){20}}

\put(60,5){\mbox{Figure 3.1.}}

\put(200,30){$y_1y_2$}
\put(220,40){\circle{1}}
\put(220,40){\line(1,2){40}}

\put(250,124){$x_1x_2$}
\put(260,120){\circle{1}}
\put(260,120){\line(1,-2){40}}

\put(300,30){$z_1z_2$}
\put(300,40){\circle{1}}
\put(300,40){\line(-1,0){80}}

\put(216,80){$x_1y_1$}
\put(240,80){\circle{1}}
\put(240,80){\line(1,0){40}}

\put(282,80){$x_1z_1$}
\put(280,80){\circle{1}}
\put(280,80){\line(-1,-2){20}}

\put(256,30){$y_1z_1$}
\put(260,40){\circle{1}}
%\put(260,40){\line(-1,2){20}}

\put(260,5){\mbox{Figure 3.2.}}

\end{picture}}

\newcommand{\trekantAI}{\vskip 1cm
\begin{picture}(300,140)
\put(25,30){$y^2$}
\put(40,40){\circle{1}}
\put(40,40){\line(1,2){40}}

\put(75,124){$x^2$}
\put(80,120){\circle{1}}
\put(80,120){\line(1,-2){40}}

\put(122,30){$z^2$}
\put(120,40){\circle{1}}
\put(120,40){\line(-1,0){80}}

\put(46,80){$xy$}
\put(60,80){\circle{1}}
\put(60,80){\line(1,0){40}}

\put(105,80){$xz$}
\put(100,80){\circle{1}}
\put(100,80){\line(-1,-2){20}}

\put(80,30){$yz$}
\put(80,40){\circle{1}}
%\put(80,40){\line(-1,2){20}}

\put(60,5){\mbox{Figure 0.1.}}

\put(200,30){$y_1y_2$}
\put(220,40){\circle{1}}
\put(220,40){\line(1,2){40}}

\put(250,124){$x_1x_2$}
\put(260,120){\circle{1}}
\put(260,120){\line(1,-2){40}}

\put(300,30){$z_1z_2$}
\put(300,40){\circle{1}}
\put(300,40){\line(-1,0){80}}

\put(216,80){$x_1y_1$}
\put(240,80){\circle{1}}
\put(240,80){\line(1,0){40}}

\put(282,80){$x_1z_1$}
\put(280,80){\circle{1}}
\put(280,80){\line(-1,-2){20}}

\put(256,30){$y_1z_1$}
\put(260,40){\circle{1}}
%\put(260,40){\line(-1,2){20}}

\put(260,5){\mbox{Figure 0.2.}}

\end{picture}}

\newcommand{\trekantB}{\vskip 1cm
\begin{picture}(300,140)
\put(20,30){$y_1^\prime y_2^\prime $}
\put(40,40){\circle{1}}
\put(40,40){\line(1,2){40}}

\put(70,124){$x_1x_2$}
\put(80,120){\circle{1}}
\put(80,120){\line(1,-2){40}}

\put(120,30){$z_1z_2$}
\put(120,40){\circle{1}}
\put(120,40){\line(-1,0){80}}

\put(36,80){$x_1y_2^\prime$}
\put(60,80){\circle{1}}
\put(60,80){\line(1,0){40}}

\put(102,80){$x_1z_2$}
\put(100,80){\circle{1}}
\put(100,80){\line(-1,-2){20}}

\put(76,30){$y_1^\prime z_2$}
\put(80,40){\circle{1}}
%\put(80,40){\line(-1,2){20}}

\put(60,5){\mbox{Figure 3.3.}}

\put(190,30){$y_1y_2y_1^\prime y_2^\prime$}
\put(220,40){\circle{1}}
\put(220,40){\line(1,2){40}}

\put(250,124){$x_1x_2$}
\put(260,120){\circle{1}}
\put(260,120){\line(1,-2){40}}

\put(300,30){$z z^\prime$}
\put(300,40){\circle{1}}
\put(300,40){\line(-1,0){80}}

\put(206,80){$x_1y_1y_2^\prime$}
\put(240,80){\circle{1}}
\put(240,80){\line(1,0){40}}

\put(282,80){$x_1z$}
\put(280,80){\circle{1}}
\put(280,80){\line(-1,-2){20}}

\put(248,30){$y_1y_1^\prime z$}
\put(260,40){\circle{1}}
%\put(260,40){\line(-1,2){20}}

\put(240,5){\mbox{Figure 3.4.}}

\end{picture}}

\newcommand{\firkanthex}{\vskip 1cm
\begin{picture}(300,140)

\put(160,80){\bigcircle{114}}

\put(115,30){24}
\put(120,40){\circle{1}}
\put(120,40){\line(1,0){80}}

\put(195,30){14}
\put(200,40){\circle{1}}
\put(200,40){\line(0,1){80}}

\put(195,122){15}
\put(200,120){\circle{1}}
\put(200,120){\line(-1,0){80}}

\put(115,122){25}
\put(120,120){\circle{1}}
\put(120,120){\line(0,-1){80}}

\put(155,30){46}
\put(160,40){\circle{1}}
\put(160,40){\line(0,1){80}}
\put(160,120){\circle{1}}
\put(155,122){35}

\put(108,72){26}
\put(120,80){\circle{1}}
\put(120,80){\line(1,0){80}}
\put(160,80){\circle{1}}
\put(148,72){36}
\put(200,80){\circle{1}}
\put(202,72){13}

\put(140,5){\mbox{Figure 4.4.}}
\end{picture}}

\newcommand{\bipullA}{\vskip 1cm
\begin{picture}(300,140)

\put(53,77){1}
\put(60,80){\circle{1}}
\put(60,80){\line(1,1){40}}
\dottedline{6}(60,80)(120,80)
\put(60,80){\line(1,-3){20}}
\put(60,80){\line(1,-1){20}}

\put(95,122){4}
\put(100,120){\circle{1}}
\put(100,120){\line(1,-2){20}}
\put(100,120){\line(-1,-3){20}}

\put(70,18){5}
\put(80,20){\circle{1}}
\put(80,20){\line(0,1){40}}
\put(80,20){\line(2,3){40}}

\put(122,82){3}
\put(120,80){\circle{1}}
\put(120,80){\line(-2,-1){40}}

\put(84,55){2}
\put(80,60){\circle{1}}

\put(70,5){\mbox{Figure 4.5.}}

\put(260,80){\bigcircle{114}}

\put(215,30){56}
\put(220,40){\circle{1}}
\put(220,40){\line(1,0){80}}

\put(295,30){67}
\put(300,40){\circle{1}}
\put(300,40){\line(0,1){80}}

\put(295,122){47}
\put(300,120){\circle{1}}
\put(300,120){\line(-1,0){80}}

\put(215,122){45}
\put(220,120){\circle{1}}
\put(220,120){\line(0,-1){80}}

\put(255,30){36}
\put(260,40){\circle{1}}
\put(260,40){\line(0,1){80}}
\put(260,120){\circle{1}}
\put(255,122){124}

\put(208,72){235}
\put(220,80){\circle{1}}
\put(220,80){\line(1,0){80}}
\put(260,80){\circle{1}}
\put(248,72){123}
\put(300,80){\circle{1}}
\put(302,72){17}

\put(240,5){\mbox{Figure 4.6.}}
\end{picture}}

\newcommand{\bipullB}{\vskip 1cm
\begin{picture}(300,140)

\put(100,80){\bigcircle{114}}

\put(55,30){57}
\put(60,40){\circle{1}}
\put(60,40){\line(1,0){80}}

\put(135,30){17}
\put(140,40){\circle{1}}
\put(140,40){\line(0,1){80}}

\put(135,122){124}
\put(140,120){\circle{1}}
\put(140,120){\line(-1,0){80}}

\put(55,122){45}
\put(60,120){\circle{1}}
\put(60,120){\line(0,-1){80}}

\put(95,30){67}
\put(100,40){\circle{1}}
\put(100,40){\line(0,1){80}}
\put(100,120){\circle{1}}
\put(95,122){234}

\put(48,72){56}
\put(60,80){\circle{1}}
\put(60,80){\line(1,0){80}}
\put(100,80){\circle{1}}
\put(88,72){36}
\put(140,80){\circle{1}}
\put(142,72){123}

\put(80,5){\mbox{Figure 4.7.}}

\put(260,80){\bigcircle{114}}

\put(215,30){67}
\put(220,40){\circle{1}}
\put(220,40){\line(1,0){80}}

\put(295,30){56}
\put(300,40){\circle{1}}
\put(300,40){\line(0,1){80}}

\put(295,122){45}
\put(300,120){\circle{1}}
\put(300,120){\line(-1,0){80}}

\put(215,122){47}
\put(220,120){\circle{1}}
\put(220,120){\line(0,-1){80}}

\put(255,30){36}
\put(260,40){\circle{1}}
\put(260,40){\line(0,1){80}}
\put(260,120){\circle{1}}
\put(255,122){124}

\put(208,72){37}
\put(220,80){\circle{1}}
\put(220,80){\line(1,0){80}}
\put(260,80){\circle{1}}
\put(248,72){123}
\put(300,80){\circle{1}}
\put(302,72){125}

\put(240,5){\mbox{Figure 4.8.}}
\end{picture}}

\newcommand{\pentpyramid}{\vskip 1cm
\begin{picture}(300,140)

\put(126,60){40}
\put(140,60){\circle{1}}
\put(140,60){\line(2,-1){40}}

\put(175,30){34}
\put(180,40){\circle{1}}
\put(180,40){\line(2,1){40}}

\put(222,60){23}
\put(220,60){\circle{1}}
\put(220,60){\line(0,1){40}}

\put(222,100){12}
\put(220,100){\circle{1}}
\put(220,100){\line(-2,1){40}}

\put(175,122){56}
\put(180,120){\circle{1}}
\put(180,120){\line(2,-3){40}}
\put(180,120){\line(0,-1){80}}
\put(180,120){\line(-2,-3){40}}
\put(180,120){\line(-2,-1){40}}

\put(126,100){01}
\put(140,100){\circle{1}}
\put(140,100){\line(0,-1){40}}
\dottedline{6}(140,100)(220,100)

\put(160,5){\mbox{Figure 4.1.}}
\end{picture}}

\newcommand{\elongpyramid}{\vskip 1cm
\begin{picture}(300,140)

\put(126,40){47}
\put(140,40){\circle{1}}
\put(140,40){\line(2,-1){40}}
\dottedline{5}(140,40)(200,40)

\put(175,12){57}
\put(180,20){\circle{1}}
\put(180,20){\line(1,1){20}}
\put(180,20){\line(0,1){60}}

\put(202,40){67}
\put(200,40){\circle{1}}
\put(200,40){\line(0,1){60}}

\put(202,100){36}
\put(200,100){\circle{1}}
\put(200,100){\line(-1,-1){20}}
\put(200,100){\line(-3,2){30}}

\put(166,72){25}
\put(180,80){\circle{1}}
\put(180,80){\line(-1,4){10}}
\put(180,80){\line(-2,1){40}}

\put(126,100){14}
\put(140,100){\circle{1}}
\put(140,100){\line(0,-1){60}}
\dottedline{5}(140,100)(200,100)

\put(165,122){123}
\put(170,120){\circle{1}}
\put(170,120){\line(-3,-2){30}}

\put(160,0){\mbox{Figure 4.2.}}
\end{picture}}

\newcommand{\selvdualpol}{\vskip 1cm 
\begin{picture}(300,140)

\put(160,80){\bigcircle{114}}

\put(115,30){6}
\put(120,40){\circle{1}}
\put(120,40){\line(1,0){80}}

\put(195,30){4}
\put(200,40){\circle{1}}
\put(200,40){\line(0,1){80}}

\put(195,122){2}
\put(200,120){\circle{1}}
\put(200,120){\line(-1,0){80}}

\put(115,122){0}
\put(120,120){\circle{1}}
\put(120,120){\line(0,-1){80}}

\put(155,30){5}
\put(160,40){\circle{1}}
\put(160,40){\line(0,1){80}}
\put(160,120){\circle{1}}
\put(155,122){1}

\put(108,72){7}
\put(120,80){\circle{1}}
\put(120,80){\line(1,0){80}}
\put(160,80){\circle{1}}
\put(148,72){c}
\put(200,80){\circle{1}}
\put(202,72){3}

\put(140,5){\mbox{Figure 4.3.}}
\end{picture}}

\newcommand{\MaFigNog}{\hskip 1cm
\begin{picture}(200,105)
\put(40,4){$x_1^{a+b+1}$} \put(56,18){x}
\dottedline{6}(70,20)(94,20)
\put(104,18){x}
\put(116,18){x}
\put(131,21){\circle{6}}
\dottedline{6}(140,20)(150,20)
\put(156,21){\circle{6}}
\put(164,4){$x_2^{a+b+1}$}

\put(64,30){x}
\dottedline{6}(76,32)(92,32)
\put(96,30){x}
\put(108,30){x}
\put(123,33){\circle{6}}
\dottedline{6}(130,32)(140,32)
\put(148,33){\circle{6}}

\dottedline{6}(80,50)(90,65)
\dottedline{6}(130,50)(120,65)

\put(105,78){\circle{6}}
\put(101,88){$x_3^{a+b+1}$}

\end{picture}
}

\title [Cellular resolutions of Cohen-Macaulay monomial ideals]
{Cellular resolutions of Cohen-Macaulay monomial ideals}
\author { Gunnar Fl{\o}ystad}
%\footnotetext{2000 {\it Mathematics Subject Classification.} 
%Primary 13D02.
%Secondary 13F55, 05E99.}
\address{ Dep. of Mathematics\\
          Johs. Brunsgt. 12\\
          5008 Bergen\\
          Norway}   
        
\email{ gunnar@mi.uib.no}

\begin{abstract}
We investigate monomial labellings on cell complexes, giving
a minimal cellular resolution of the ideal generated by these monomials, 
and such that the associated quotient ring is Cohen-Macaulay.  
We introduce a notion of such a labelling being maximal. There is only
a finite number of maximal such labellings for each cell complex,
and we classify these for trees, subdivisions of polygons, and some
classes of selfdual polytopes. 
\end{abstract}

\maketitle

{\Small 2000 MSC : Primary 13D02. Secondary 13F55, 05E99.}

\section*{Introduction}

%In this paper we introduce the notion of maximal monomial ideals and
%classify them in some basic cases.

 In this paper we study cellular resolutions of monomial ideals which have a Cohen-Macaulay quotient
ring. Cellular resolutions of monomial ideals, introduced in \cite{BPS} and \cite{BS}, is a very natural
technique for constructing resolutions of monomial ideals, and appealing in its 
blending of topological constructions, 
combinatorics and algebraic ideas. Much activity has centred around it in the last decade, and
good introductions and surveys
may be found in \cite{MiSt} and \cite{We}. 
Usually one starts with a monomial ideal and finds a suitable labelled cell complex giving
a (preferably minimal) resolution of the monomial ideal. 
It was hoped that a minimal resolution of a monomial 
ideal was always cellular, but this was shown recently not to be so, \cite{Ve}.

Here we turn this around and start with the cell complex, and ask what monomial labellings
are such that this cell complex gives a minimal cellular resolution of the ideal formed by the
monomials in the labelling. To limit the task we assume that the monomial labelling is such that
the monomial quotient ring is Cohen-Macaulay, and the cell complex gives a minimal cellular resolution of
it. Such a labelling will be called a Cohen-Macaulay (CM) monomial labelling.

For a given cell complex, we define a notion of maximal CM monomial labelling. These are 
essentially labellings by monomials ${\bf x}^{{\bf a}_i}$ in a polynomial ring 
$\kr[x_1, \ldots, x_n]$ such that any CM monomial labelling of the cell complex by monomials
${\bf y}^{{\bf b}_i}$ in $\kr[y_1, \ldots, y_m]$ may be obtained by a multigraded homomorphism
$\kr[x_1, \ldots, x_n] \pil \kr[y_1, \ldots, y_m]$ sending the monomial ${\bf x}^{{\bf a}_i}$
to ${\bf y}^{{\bf b}_i}$. For any cell complex there turns out to be a finite number of 
maximal CM monomial labellings, and we are in particular concerned with classifying these labellings.

First we consider the case where the cell complex is one-dimensional, it must then be a tree.
We show that any CM monomial quotient ring of codimension two has a cellular resolution given 
by a tree. Then we show that for a given tree there is a unique maximal CM monomial labelling
up to isomorphism.

Then we consider the case where the cell complex is two-dimensional. First we look at the case
of a polygon. If it is an $n$-gon with $n$ even, there are no CM monomial labellings, and if 
$n$ is odd there is a unique CM monomial labelling consisting of monomials of degree $(n-1)/2$
in $n$ variables. (This is known but we do not know of a specific reference.) 
We then proceed to consider
subdivisions of polygons. By the techniques we use this is quite hard and we only do this in 
the case of a polygon with a single chord. We show that there are then two maximal CM 
monomial labellings. The description of them splits into the cases of whether we have an $n$-gon, 
where $n$
is even or odd. In all these cases the monomials are in $n+1$ variables. This 
makes it reasonable to conjecture that in a subdivision of an $n$-gon with $r$ chords, any
maximal CM monomial labelling consists of monomials in  $n+r$ variables.

An interesting example is the subdivision of the hexagon. 

\trekantAI
\medskip

\noindent A CM monomial labelling is given by Figure 0.1. One may polarise this and get a CM monomial labelling
in six variables, Figure 0.2. However this is not maximal. 
A maximal monomial labelling is given by Figure 3.4 in Subsection \ref{SecPolyodd}, and
consists of monomials in eight variables.

In the end we consider CM monomial labellings of polytopes of dimension three and larger.
We classify the maximal CM monomial labelling on pyramids over self-dual polytopes $X$, provided
we know the maximal CM labellings of $X$. We also consider the elongated pyramid over $X$ which is
a union of $X \times [0,1]$ and the pyramid over $X$ glued together at $X \times \{1\}$.
Given a maximal CM labelling of $X$, we construct such a labelling over the elongated pyramid.
We also give several examples of CM labellings of three-dimensional self-dual polytopes, which
give cellular resolutions of Gorenstein Stanley-Reisner rings of codimension four.

\vskip 3mm

The organisation of the paper is as follows. In Section 1 we define the notion of a maximal 
CM monomial labelling. We show that there is a finite number of such for any cell complex, and we
give a topological characterisation of such labellings. 
In Section 2, 3, and 4 we consider maximal CM monomial labellings of cell complexes of dimension
1, 2, and 3 and higher, respectively. In Section 2 we consider the case of trees, and show that
there is a unique maximal CM monomial labelling, up to isomorphism. In Section 3 we consider the
case of subdivisions of polygons, and in Section 4 we give maximal CM monomial labellings
of self-dual polytopes, as well as examples of monomial labellings of three-dimensional 
self-dual polytopes giving cellular resolutions of Gorenstein Stanley-Reisner rings of codimension four.

\section{Maximal Cohen-Macaulay monomial labellings}

Let $\kr [x_1, \ldots, x_r]$ be a polynomial ring, which we may identify
with the semi-group ring $\kr [\NN^r]$. Given an integer $n$. We shall
consider ordered sets of monomials $(\xx^{\aa_1},\xx^{\aa_2}, \ldots, \xx^{\aa_n})$
where none divide any other, i.e. they form a set of minimal generators for
an ideal.

A semi-group homomorphism $\NN^r \mto{\phi} \NN^s$ maps this ordered set of monomials
to another ordered set $(\yy^{\bb_1},\yy^{\bb_2}, \ldots, \yy^{\bb_n})$ given by
$\phi(\aa_i) = \bb_i$. In this way we get a category $\Mon(n)$ whose objects are
pairs $(\NN^r,\aa)$ where $\aa$ is an $n$-tuple of elements of $\NN^r$
and morphisms are given by semi-group homomorphisms as above, mapping
the $n$-tuples to each other.

Now consider the full subcategory $\CM(n,c)$ of $\Mon(n)$ consisting of those 
ordered sets of monomials generating an ideal $I$ such that the quotient
ring $\kr [x_1, \ldots, x_r]/I$
is a Cohen-Macaulay ring of codimension $c$.

\eks Monomial ideals which are not square free may be polarised. For
instance $(a^2, ab, b^2)$ is in $\CM(3,2)$ and polarises to
$(a_1a_2, a_1b_1, b_1b_2)$ also in $\CM(3,2)$ (with $a_1-a_2, b_1-b_2$ as
a regular sequence in the quotient ring). This leads us to think 
of non square free monomial ideals as somewhat compressed monomial ideals.
Alternatively polarisation is a ``loosening up'' of the non square free
monomial ideal. However, also square free monomial ideals can be ``loosened
up''. For instance $(ca,ab,bc)$ in $\CM(3,2)$ may be ``loosened up'' to
$(c_1a,ab,bc_2)$ isomorphic to the polarization above (here $c_1-c_2$ is
a regulare element in the quotient ring). A central theme of this
paper is to investigate the most ``free'' or ``loosened up'' monomial
ideals. We term these maximal monomial ideals. Here is the formal 
definition.
\eksfin

%\begin{lemma} If $F$ is a free resolution of the ideal generated by the 
%$\xx^{\aa_i}$, and $\NN^r \pil \NN^s$ a semi-group homomorphism, 
%then 
%\[ F \te \kr_{\NN^r} [\NN^s] \]
%is a free resolution of the ideal generated generated by the $\yy^{\bb_i}$'s.
%\end{lemma}

\defi
An object $(\NN^r, \aa)$ in $\CM(n,c)$ is {\it maximal}
if whenever there is a morphism $\phi : (\NN^s, \bb) \pil (\NN^r, \aa)$, the map $\NN^s \pil
\NN^r$ is a surjection and there is a
splitting $\psi : \NN^r \pil \NN^s$, i.e. $\phi \circ \psi$ is the identity on $\NN^r$.
\defifin

\eks The pair $(\NN^n, (x_1, \ldots, x_n))$ is maximal in $\CM(n,n)$
and is, up to isomorphism, i.e. permutation of variables, the only such object.
For, instance, the pair $(\NN^3, (x_1, x_2x_3))$ is not maximal in $CM(2,2)$ because there is a 
morphism  $\phi$ to it
from $(\NN^2, (x_1, x_2))$ such that $e_1 \mapsto e_1, e_2 \mapsto e_2 + e_3 $. Note that
there is also a morphism from $(\NN^3, (x_1, x_2x_3))$ to $(\NN^2, (x_1, x_2))$ sending
$e_1 \mapsto e_1, e_2 \mapsto e_2$, and $e_3 \mapsto 0$. The morphism $\phi$ is a splitting
of it, consistent with $(\NN^2, (x_1, x_2))$ being maximal.
\eksfin

\rem Another paper that considers maps of monomial generators is \cite{GPW}. There one studies
the LCM lattice of the monomials and considers a map $\phi$ between two such lattices 
which induces an isomorphism 
on the atoms i.e. the monomial generators. Note that this is a somewhat different situation from 
ours since in our case, the map on monomials is induced from a map of semigroup rings. 
If the map $\phi$ they consider preserves joins,
they show that if $F$ is a free resolution of the ideal generated by the first set of monomials, 
there is a construction of a complex $\phi(F)$ which is a free resolution of the of the ideal 
generated by the second set of monomials. If $\phi$ is an isomorphism of lattices, then
$F$ is a minimal resolution iff $\phi(F)$ is a minimal resolution.
\remfin

\begin{lemma} Let $(\NN^r, \aa)$ and $(\NN^s, \bb)$ be two maximal elements in 
$\CM(n,c)$. Then they are either isomorphic or there are no morphisms between them.
\end{lemma}

\begin{proof} If $(\NN^s, \bb) \pil (\NN^r, \aa)$ is a morphism, then due to the maximality
of $(\NN^r,\aa)$, there is a splitting $(\NN^r, \aa) \pil (\NN^s, \bb)$, so $s \geq r$. But 
similarly since $(\NN^s,\bb)$ is maximal, we must have $s \leq r$. So $s=r$ and the semigroup
homomorphism $\NN^r \pil \NN^s$ is an isomorphism.
\end{proof}

\begin{proposition} \label{MaxProFin} In $\CM(n,c)$ there is a finite set of maximal objects, each of 
which consists of square free monomials.
To any object $(\NN^r, \aa)$ in $\CM(n,c)$ there is a morphism from some maximal object to this
object, i.e. there is a maximal object $(\NN^s,\bb)$ and a semi-group homomorphism
$\NN^s \pil \NN^r$ taking $\bb$ to $\aa$.
\end{proposition}

\begin{proof}
If the monomials are not square free, we can polarise the monomials. So we
get a morphism $(\NN^s, \bb) \pil (\NN^r,\aa)$, where the $\bb$'s 
are $0,1$-vectors. Clearly there cannot be any splitting $\phi$ in the reverse
direction if $\aa_i$ is not square free, since then $\phi(\aa_i)$ would not
be either.
Thus all maximal objects must be square free. 

Now given an ordered set of square-free monomials $(m_1, \ldots, m_n)$
in $\kr [x_1, \ldots, x_r]$.
To each variable $x_p$ we associate the subset $V_p$ of $[n]$ consisting of those
positions $i$ such that $x_p$ divides $m_i$. This gives us a multiset of subsets
of $[n]$, and this multiset determines the isomorphism class in $\CM(n,c)$ of the ordered
set of monomials. If $V_p = V_q$ for some $p < q$, we get a morphism from some
$(\NN^{r-1}, \bb)$ to $(\NN^r, \aa)$ by sending $e_i$ to $e_i$ for $i \neq p,q$ and
sending $e_p$ to $e_p + e_q$. But then $(\NN^r, \aa)$ cannot be maximal (there 
cannot be a splitting due to the ranks of semigroups). Iterating this process
we can in the end assume that we to our monomial labelling have associated a 
family of distinct subsets of $[n]$. A maximal object must be of this kind.
Since there is only a finite number of families of subsets of $[n]$, there
is only a finite number of maximal objects.

\end{proof}

If $m_1, \ldots, m_n$ are square free monomials, we may to each variable $x_p$ 
associate the set $V_p$ of all $i$ in $[n]$ such that $x_p$ divides $m_i$.
If the monomials give a maximal object, we know from the proof above
that the $V_p$ are all distinct, thus forming a family of subsets of $[n]$.
We let $\CM_*(n,c)$ denote the full subcategory of $\CM(n,c)$ consisting
of $(\NN^r, \aa)$ such that the monomials $\xx^{\aa_i}$ are square free
and the subsets $V_p \sus [n]$ associated to the variables are all distinct.
Note that this family of subsets determines the isomorphism class of the object
$(\NN^r, \aa)$. Also if $(\NN^r, \aa)$ and $(\NN^s, \bb)$ are objects
in $\CM_*(n,c)$ with associated families $\gF$ and $\gG$ of subsets of $[n]$, 
then there is a morphism from the first
to the latter iff every element of $\gG$ is a disjoint union of elements of $\gF$.
This lead us to on the families of subsets of $[n]$ to consider the refinement partial
order given by $\gF \succ \gG$ iff $\gF$ consists of refinements of elements of $\gG$
together with additional subsets of $[n]$.
(A refinement of a set $S$ are subsets of it such that $S$ is a disjoint union of them.)

If $\gF$ is a family of subsets of $[n]$ we let its {\it reduction}
$\gF^{\red}$ be the subfamily of $\gF$ consisting of those elements 
(which are subsets of $[n]$)
which are not disjoint unions of other elements of $\gF$.

\begin{proposition}  \label{MaxProMax}
a. If $\gF$ corresponds to an object in $CM_*(n,c)$, then $\gF^{\red}$ corresponds to  an 
object in this category.

b. An object in $\CM_*(n,c)$ is maximal iff the associated family $\gF$ is reduced and
is maximal among reduced associated families for the refinement order.
\end{proposition}

\begin{proof}
%Let $(\NN^s, \bb)$ and $(\NN^r, \aa)$ be objects in $\CM_*(n,c)$ whose associated
%families of subsets are $\gG$, resp. $\gF$ with cardinalities $s$ resp. $r$.

%There is a morphism $(\NN^s, \bb) \pil (\NN^r, \aa)$ if and only if
%any subset of $\gF$ is the disjoint unison of subsets of $\gG$, i.e.
%$\gF^{\red} \sus \gG^{\red}$.

%If $(\NN^r, \aa)$ is maximal the existence of a splitting then implies
%$\gG^\red \sus \gF^\red$, so we have $\gF^\red = \gG^\red$.
%From this we conclude that $(\NN^r, \aa)$ is maximal it there are no
%$(\NN^s,\bb)$ whose associated $\gG$ has $\gG^\red \supsetneq \gF^\red$,
%and among those $(\NN^s, \bb)$ whose associated $\gG$ have 
%$\gG^\red  = \gF^\red$, then $\gF$ has minimal cardinality, i.e.
%$\gF$ must be reduced.

a. Let $\gF$ correspond to $(\NN^r, \aa)$. The elements of $\gF$ are indexed by basis elements
$e_i$ of $\NN^r$. Let $\gF \backslash \gF^{\red}$ consist of the sets $S_{t+1}, \ldots S_{r}$
corresponding to $e_{t+1},\ldots, e_{r}$, so $\NN^r = \NN^t \bigoplus \oplus_{i = t+1}^r \NN e_{i}$. 
Then $\gF^{\red}$ corresponds to the monomials  $\bb_i$ we get as the images of 
$\aa_i$ by the projection $\NN^r \pil \NN^t$.
Alternatively the ring $\kr[x_1, \ldots, x_t] / (\xx^{\bb_1}, \ldots, \xx^{\bb_r})$ is obtained from 
$\kr[x_1, \ldots, x_r]/(\xx^{\aa_1}, \ldots, \xx^{\aa_r})$ by dividing out by $x_i - 1$ for 
$i = t+1, \ldots, r$. Now the codimension of the latter ring is the minimal number of sets in $\gF$
covering $[n]$. Similarly the codimension of the former ring is the minimal number of sets in 
$\gF^{\red}$ covering $[n]$. But the codimension of the former ring is greater or equal to that of 
the latter ring since $\gF^{\red} \sus \gF$. Since the latter ring is Cohen-Macaulay, their 
codimensions must in fact be equal, 
and by \cite[Prop.18.13]{Ei}, the first is also Cohen-Macaulay. Thus
$\gF^{\red}$ corresponds to an object in $\CM_*(n,c)$. 

b. Note that we have a morphism $(\NN^t, \bb) \pil (\NN^r, \aa)$ 
since each element of $\gF$ is a disjoint union of elements of $\gF^{\red}$.
Thus if $\gF$ is maximal it must be equal to $\gF^{\red}$. 
Clearly then $(\NN^r, \aa)$ is maximal iff the associated family $\gF$ is maximal
among reduced associated families for the refinement order.
\end{proof}

%Note that if $F \sus G$ are the families associated to $(\NN^r,\aa)$ and
%$(\NN^s,\bb)$ there is morphism from the latter object to the former.
%Now let $F_1, \ldots, F_m$ be the maximal families under inclusion and let
%$(\NN^{r_i}, \aa_i)$ be the object associated to $F_i^{\text{red}}$.
%We shall show that this is maximal. But to give a morphism 
%$(\NN^s, \bb) \pil (\NN^{r}, \aa)$ in $\CM_*(n,c)$ whose associated families are $G$ and $F$
%means that each subset of $F$ is a disjoint union of subsets of $G$, i.e. 
%$F^\red \sus G^\red$. If there is a morphism in the other direction, we must then have
%$F^\red = G^\red$. Thus if $(\NN^r,\aa)$ is maximal, we see that $F$ mus be the reduction
%of a maximal family of subset.

Now we shall consider some subcategories of $\CM(n,c)$.
First let $X$ be a regular cell complex (see \cite{BH} for definition) of dimension
$d = c-1$, where the vertices
are labeled by elements of $[n] = \{1,2, \ldots, n\}$, i.e. they are ordered.
Let $\CM(X)$ be the subcategory of $\CM(n,c)$ consisting of all objects such that when
the vertices of $X$ are labelled with the monomials in this object, the cellular complex
associated to this monomial labelling gives a minimal free resolution of the ideal 
generated by these monomials. Such a labelling will be called a Cohen-Macaulay (CM)
labelling of $X$.

\medskip
%Another subcategory of $\CM(n,c)$ may be obtained by considering an object $(\NN^r,\aa)$
%here, and forming the subcategory $\CM^c(\NN^r,\aa)$ (or just $\CM^c(\aa)$ 
%consisting of all pairs $(\NN^s,\bb)$ which have a morphism to $(\NN^r,\aa)$.

\begin{proposition} In $\CM(X)$ there is a finite set of maximal 
objects. These objects lie in the subcategory $\CM_*(X)$,
which is the intersection of $\CM(X)$ and $\CM_*(n,c)$.
%   For each object in $\CM(X)$ (resp. $\CM(\aa)$) there is a morphism from some
%maximal object in the category to the given object. If the object is in
%$\CM^c_*(X)$ (resp. $\CM^c_*(\aa)$), this morphism from a maximal object is unique.
\end{proposition}

\begin{proof} 
This goes completely as the proof of Proposition \ref{MaxProFin}.
\end{proof}
%To show uniqueness, an object $B$ in $\CM_*(X)$ corresponds to a family of subset $\gF$
%of the vertices. If $\gG$ corresponds to a maximal object $A$ and there is a 
%morphism $A \pil B$, this means that each $T$ in $\gF$ is a disjoint unions of 
%subsets of $\gG$. The question is then wether for each $T$ this union is unique
%up to order.
%\end{proof}

\rem Another variant is to fix an object $A = (\NN^r, \aa)$ in 
$\CM(n,c)$ and define $\CM(A)$ to be all objects $B$ in $\CM(n,c)$
which has a map $B \pil A$. In \cite{FlVa} we consider the case
when $A$ consists of all square free monomials of degree $d$ in $m$ variables.
This is an object of $\CM(\binom{m}{d}, m-d+1)$. Conjecture
1, in Section 4 in \cite{FlVa} may be formulated as 
saying that every maximal object
over $A$ consist of monomials in $dm-2\binom{d}{2}$ variables or less. 
We showed that
this number of variables may be attained. The ideas implicit in 
this conjecture was a motivating factor for this paper.
Conjeture \ref{PolyConjVar} in the present paper has a similar flavor.
\remfin

To an isomorphism classes of objects in $\CM_*(X)$ there is associated a family $\gF$
of subsets of $[n]$ which determines this isomorphism class. In order for a family of subsets
of $[n]$ to correspond to an object of 
$\CM_*(X)$ some conditions must be fulfilled.

\begin{proposition} \label{MaxPropFamKrit} Let $\dim X = d$.
A family of subsets $\gF$ of $[n]$ corresponds to an object in $\CM_*(X)$ iff the 
following conditions hold.
\begin{itemize}
\item[1.] No $d$ of the subsets in $\gF$ cover $[n]$.
\item[2.] Let $W$ be a union of subsets of $\gF$. Then the restriction of $X$ to the
complement of $W$ is acyclic.
\item[3.] For every pair $F \susneq G$ of (vertices of ) faces of $X$, there is an $S$ in $\gF$ such that
$S \cap F$ is empty, but $S \cap G$ is nonempty.
\end{itemize}
\end{proposition}

\rem Letting $F$ be the empty set and $G$ consist of a single vertex $v$ in 
condition 3., we see that the elements of $\gF$ cover $[n]$.
\remfin

\rem In brief condition 2. shows that $X$ gives a cellular resolution of the ideal, condition 3.
shows the minimality of this resolution, and condition 1.
(together with the fact that the elements of  $\gF$
cover $[n]$)
shows that the ideal has codimension 
$\geq \dim X + 1$. Thus condition 1. and 2. gives that the monomial quotient ring is Cohen-Macaulay.
\remfin

\begin{proof} We first show that condition 2. holds if and only if $X$ gives a cellular resolution of the 
ideal associated to the monomial labelling. The latter is equivalent to the subcomplex 
$X_{\leq \bb}$,
induced on the vertices corresponding 
to monomials $\xx^\aa$ with $\aa \leq \bb$, being acyclic for every $\bb$.

Suppose now condition 2. holds.
Then $X_{\leq \bb}$ is $X$ restricted to the set $U$ of vertices $i$ such that $m_i$
divides $\xx^\bb$. If there is a zero in position $p$ in $\bb$, then clearly $V_p$
is disjoint from $U$. So all such $V_p$ are subsets of the complement $W = \overline{U}$.
%If there were some $q$ in $\overline{V} \backslash \cup V_p$, then $m_q$ divides
%$\xx^\bb$. sot there is some variable $x_p$ in $m_q$ not occuring in $\xx^\bb$.
%But the union of these must be all of 
But the union of these must be all of $W$, since if $q$ is in $W$
then $m_q$ does not divide $\xx^\bb$ and so there must be some variable $x_p$ in $m_q$
not in $\xx^\bb$, and so $q$ is in $V_p$. Thus $X_{\leq \bb}$ is $X$ restricted to the 
complement of a union of $V_p$'s, and so is acyclic.

Now suppose $X_{\leq \bb}$ is always acyclic.
If $W$ is a union of $V_p$'s, let $\bb$ be the $0,1$-vector with $0$ in 
positions $p$. Then $X$ restricted to the complement of $W$ is $X_{\leq \bb}$, and so acyclic.

\medskip Now consider condition 1. That the monomials in the labelling generate an
ideal $I$ of codimension $\geq \dim X + 1$ is equivalent to there being no $\dim X$ 
variables whose associated vertex sets cover the vertices of $X$.

\medskip Condition 3. gives the condition of minimality of the cellular resolution.
In fact, minimality is equivalent to the fact that for each pair $F \subsetneq G$
the monomial label $\xx^\bb$ associated to $G$ is strictly larger than the 
monomial labelling $\xx^\aa$ associated to $F$. Since we are considering 
square free monomials, some variable $x_p$ must occur in $\xx^\bb$ and not
in $\xx^\aa$. Hence $V_p \cap G$ is nonempty while $V_p \cap F$ is empty.
\end{proof}

An extra condition that must be fulfilled if the family $\gF$ corresponds to 
a maximal object is the following.

\begin{lemma} \label{MaxLemXbet} If a family of subsets $\gF$ of $[n]$ corresponds 
to a maximal object in $\CM_*(X)$, then for every $S$ in $\gF$, 
the restriction of $X$ to $S$ is connected.
\end{lemma}

\begin{proof} Suppose $X$ restricted to $S$ is not connected, and let $S$ be $S_1 \cup S_2$
such that $X_{|S}$ is the disjoint union of $X_{|S_1}$ and $X_{|S_2}$.
We want to show that $\gF^\prime = \gF \cup \{S_1, S_2\}$ fulfils the criteria of 
Proposition \ref{MaxPropFamKrit}. 
%But this would either show that $\gF$ is not reduced, or 
But then $\gF^{\prime \red}$ would give us a 
larger family of subsets for the refinement order, 
contradicting the fact that $\gF$ corresponds to a maximal object.

%But since $\gF$ is reduced this would contradict
%the maximality of $\gF$.
The criteria 1. and 3. hold for $\gF^\prime$ given that they hold for $\gF$. We must show
that 2. holds. Let $\gG$ be the set of complements of sets in $\gF$, 
and let $T$ be the intersection
of elements in a subfamily of $\gG$. Let $T_1$ and $T_2$ be the complements
$\overline{S_1}$ and
$\overline{S_2}$ respectively. We know that $X$ restricted to $T$ and to $T \cap T_1 \cap T_2$
are acyclic. We must show that $X$ restricted to $T \cap T_1 $ and to $T \cap T_2$ is
acyclic. This follows from the following.

\begin{claim} Suppose $Y_1$ and $Y_2$ are open subsets of $Y$ such that $Y = Y_1 \cup Y_2$ 
and $Y_1 \cap Y_2$ are acyclic. Then $Y_1$ and $Y_2$ are acyclic.
\end{claim}

This claim follows form the Mayer-Vietoris sequence.
\end{proof}

We also have the following property of a maximal family.

\begin{proposition} \label{MaxPropT}
Let $\gF$ be a maximal family in $CM_*(X)$. Let $t\in T \in \gF$.
Then there exists $S_1, \ldots, S_{\dim X}$ in $\gF$ such that

1. $T \cup \cup_{i=1}^{\dim X} S_i$ covers $X$,

2. $t$ is not in any $S_i$.
\end{proposition}

\begin{proof}
By Proposition \ref{MaxPropFamKrit}.3 
there are $S_1, \ldots, S_r$ in $\gF$ whose union contains all the
neighbour vertices of $t$, but not $t$ itself.

Note that since the complement of $S_1 \cup \ldots \cup S_r$ must be
connected, this complement is simply $\{t\}$ and so $S_1, \ldots, S_r$
cover $X\backslash \{ t \}$. 

Let $\gG$ be the family consisting of the $S_i$ and $T$. Then $X$ restricted
to every complement of a union of elements of $\gG$ is acyclic. Hence
the associated monomial labelling of $X$ gives a cellular resolution of the
associated ideal. By the Auslander-Buchsbaum theorem, the corresponding quotient ring then 
has codimension $\leq \dim X + 1$. Therefore one 
must be able to cover $X$ with $\dim X + 1$ subsets in the family $\gG$,
and this cover must contain $T$ since only $T$ contains $t \in V$.
\end{proof}

We let $\CM_\dagger(X)$ be the subcategory of $\CM_*(X)$ such that the 
associated family $\gF$
also fulfils the condition of Lemma \ref{MaxLemXbet}. Then all maximal 
monomial labellings of $X$ lie in this subcategory.

%\medskip
%
%Now let $\aa_1, \ldots, \aa_r$ be the maximal objects of $\CM(\aa)$. Then there are 
%morphisms $\aa_i \pil \aa$ for each $i$. For a subset $I$ of $[r]$, let 
%$\CM^I(A)$ be the subcategory of objects $\bb$ such there are morphisms
%$\aa_i \pil \bb$ for each $i$ in $I$.

%\begin{proposition} \label{MaxProLatt} $\CM^I(A)$ has a maximal object unique up to 
%isomorphism. This object is in $\CM^c_*(A)$. 
%\end{proposition}

%\begin{proof} Let $\gF_1, \ldots, \gF_r$ be the families of subsets of $V$ corresponding to 
%the maximal objects $\aa_1, \ldots, \aa_r$. That there is a morphism from $\aa_i$ to
%an object $\bb$ corresponding to $\gG$ means that every $T$ in $\gG$ is a disjoint union
%of sets in $\gF_i$. Now given $I$ and suppose there is a morphism from $\aa_i$ to $\bb$
%for each $i$ in $I$. Take all families $\gF$ corresponding to such $\bb$ and form
%their union and its reduction. Then clearly this must be the unique maximal object in 
%$\CM_*^I(\aa)$.
%\end{proof}

%\rem  We thus get a boolean lattice where for each $I \sus [n]$ the element is the 
%maximals object in $\CM^I(\aa)$.
%\remfin

\section{Cellular resolutions of projective dimension 2}

Let an ordered set of monomials generate an ideal $I$ such that
$S/I$ is Cohen-Macaulay of codimension two. A minimal cellular
resolution of $S/I$ must then be an acyclic graph, a tree. We first show that
such a cellular resolution exists, describing in principle all
such graphs.

\subsection{Existence of cellular resolution}
 Let $m_1, \ldots, m_n$ be the monomials, and let $K$ be the complete graph
whose vertices are $[n]$. Label vertex $i$ with $m_i$ and the edge $\{i,j\}$
with $\lcm(m_i,m_j)$.
For $d \in \NN$, let $K_{\leq d}$ be the subgraph of $K$, consisting of all
vertices and edges labelled with a monomial of total degree $\leq d$.
Let $\{ F_i\}$ be a sequence of subgraphs of $K$ such that the following holds.
\begin{itemize}
\item[i.] $F_i \sus F_j$ for $i \leq j$, 
\item[ii.] $F_i$ is a spanning forest for $K_{\leq i}$. 
\end{itemize}
As soon as $K_{\leq d}$ contains all vertices and is connected, $F_i$ will be $F_d$ for
$i \geq d$ and $F_d$ is a spanning tree $T$ for $K$.

\begin{proposition} \label{TreProEks} Let $m_1, \ldots, m_n$ be generators of $I$ such that
$S/I$ is Cohen-Macaulay of codimension two.

a. The labelled tree $T$ constructed above gives a minimal cellular resolution for $I$.

b. If a tree $T$ labelled by the monomial $m_1, \ldots, m_n$ gives a 
minimal cellular resolution of $S/I$, then $T$ may be obtained by the construction above.

\end{proposition}

\begin{proof}
a. Let 
\[ \oplus_{i =1}^n Se_i \mto{d} S \pil S/I \]
be the start of the minimal resolution and let $\gK$ be the first syzygy module,
the kernel of $d$.

Since the Taylor complex, the cellular complex associated to the $n-1$-simplex
labelled by $m_1, \ldots, m_n$,
gives a resolution of $S/I$, the first syzygy module will be generated by 
\[ \sigma_{i,j} = \frac{\lcm(m_i,m_j)}{m_j} e_i - \frac{\lcm(m_i,m_j)}{m_i}e_j. \]
Each such syzygy corresponds to en edge in the complete graph $K$.

Let $e$ be the least integer for which $K_{\leq e}$ contains an edge. The edges in $F_e$
give an injective map
\[ \oplus_{\{i,j\} \in F_e} Se_{i,j} \pil  \gK \]
 It is injective because $F_e$ does not have homology in 
(homological) degree $1$. 
Also the syzygies corresponding to the edges in $F_e$ generate $\gK_{\leq e}$.
To see this let $\sigma_{i,j}$ be a syzygy, associated to an edge $\{i,j\}$ not in $F_e$, 
then $F_e \cup \{i,j\}$
will contain a cycle $i_1 \pil i_2 \pil \cdots \pil i_r \pil i_1$. Considering
$S(-e)^{r} \pil \gK$ defined by these edges and letting $M$ be the least common multiple 
of the $m_{i_j}$, we get a minimal syzygy
\[ \sum_j M/{\lcm(m_{i_j},m_{i_{j+1}})} \sigma_{i_j,i_{j+1}} = 0 \]
where the coefficient of $\sigma_{i,j}$ is nonzero.
Since $I$ has projective dimension one, some coefficient must be constant here, and since
all $\sigma_{i_j,i_{j+1}}$ have degree $e$, all coefficients must be constant. Hence 
$\sigma_{i,j}$ is a linear combination of syzygies corresponding to edges in $F_e$.

\medskip
Now we get further a map $\oplus_{\{i,j\} \in F_{e+1}} Se_{i,j} \pil \gK$. Again this map
is injective. We may also argue as above that it is surjective on $\gK_{\leq e+1}$: 
If $\{i,j\}$ is an edge in $K_{\leq e+1}$ not in $F_{e+1}$ or $K_{\leq e}$, 
then adjoining it to $F_{e+1}$ 
we get again a cycle and a syzygy.  
This must have degree $e+1$ and since $\sigma_{i,j}$ has degree $e+1$ the coefficient
of $\sigma_{i,j}$ must be a constant and so it is a linear combination of the other syzygies. 
In this way we may continue and get
that $\oplus_{\{i,j\} \in T} Se_{i,j} \pil \gK$ is injective and surjective and so an 
isomorphism.

\medskip

b. Let $T$ be a tree labelled by $m_1, \ldots, m_n$ giving a minimal cellular resolution of $S/I$. We will show
that $T_{\leq i}$ is a spanning forest for $K_{\leq i}$.

All edges of $T_{\leq i}$ are contained in $K_{\leq i}$. The uniqueness of graded Betti numbers
in a minimal free resolution, implies that the cardinality of $T_{\leq i}$ equals
the cardinality of $F_{\leq i}$ in the resolution constructed in a. Hence $T_{\leq i}$
must be a spanning forest for $K_{\leq i}$.
\end{proof}

\eks The ideal $(x^n, x^{n-1}y, \ldots, y^n)$ is of codimension two with Cohen-Macaulay
quotient ring. There is a unique tree giving a cellular resolution of this ideal,
namely the linear graph on $n+1$ vertices.
\eksfin

On the opposite side of the spectrum one has the following.

\eks \llabel{TreEksStreng}
Let $x_1, \ldots, x_n$ be the variables. For $i = 1, \ldots, n$ let $m_i$ be the monomial 
$\Pi_{p \neq i} x_p$.
These generate an ideal of codimension two whose quotient ring is Cohen-Macaulay. By
the construction in the theorem, any tree $T$ with vertices $[n]$, gives a cellular resolution
of the ideal $I$.
\eksfin

\subsection{Maximal CM monomial labellings}
Now given a tree $T$ with vertex set $[n]$, we shall show that, up to isomorphism, there is 
a unique maximal monomial labelling in $\CM(T)$. Let us describe this.

\vskip 3mm 
\noindent {\bf Maximal monomial labelling of T.}
Orientate $T$, i.e. give each
edge an orientation. Each edge $s \mto{e} t$ disconnects the tree into two parts.
For each edge $e$ associate a variable $x_e$ to all nodes in the connected component
of $s$ and $y_e$ to the connected component of $t$. 
To each node $v$ in $T$ there will now be a map $\{ \text{edges in } T \} \mto{v} \{x,y\}$.
Here $v(e) = x$ if $x_e$ is associated to $v$ and correspondingly for $y$.
We label the vertex $v$ with the product of all the variables associated to 
$v$, i.e. with  
\begin{equation} M_v = \underset{v(e) = x}{\Pi} x_e \times \underset{v(e) = y}{\Pi} y_e
\label{TreLabMv}
\end{equation}

\medskip

\begin{theorem} \label{TreTheMax} Let $T$ be a tree on the vertices $[n]$. In $\CM(T)$ there is, 
up to isomorphism, a unique maximal object $M$ given by the monomial labelling (\ref{TreLabMv}). 
Moreover, for every monomial labelling $L$ in $\CM(T)$ there is a unique morphism $M \pil L$
to this object from the maximal object.
\end{theorem}

\begin{proof} Given an element in $\CM(T)$ where the monomials $m_i$ are in the polynomial ring
$\kr [z_1, \ldots, z_r]$.
If $s \mto{e} t$ is an edge, let the reduced expression for $m_t/m_s$ be $\zz^{a_e}/\zz^{b_e}$.
The sought for morphism of semigroup rings from $\kr [\{x_e\}_{e \in E}, \{y_e\}_{e \in E}]$ to
$\kr [z_1, \ldots, z_r]$ must send $M_t/M_s = y_e/x_e$ to $m_t/m_s$. Since the images of
$y_e$ and $x_e$ must be relatively prime (otherwise all the monomials would
have a common factor), the only possibility is sending 
$y_e$ to $\zz^{a_e}$ and
$x_e$ to $\zz^{b_e}$. Hence the uniqueness of the morphism is clear.
We shall now show that this morphism actually does send the monomial $M_v$ in (\ref{TreLabMv})
to $m_v$.

Consider a variable, say $z_1$. To each vertex $v$ we associate the exponent of $z_1$ in $m_v$.
This set of $z_1$-exponents must be convex in the sense that given a path (in the non-oriented graph)
$s_1 \pil s_2 \pil \ldots
\pil s_{n-1} \pil s_n$, the $z_1$-exponents on this path from $s_1$ to $s_n$ 
must first be nonincreasing
and then nondecreasing. To see this, suppose to the contrary that there were $i < j < k$
where the exponents fulfilled $n_i < n_j > n_k$. Then letting $M$ be $\lcm(m_i, m_k)$, 
the subgraph $T_{\leq M}$ 
would not be connected, contradicting the fact that $T$ gives a cellular resolution.

Now orientate the graph so that all arrows point towards $v$. 
There will be a vertex $u$ such that $z_1$ does not occur in $m_u$ 
(otherwise the ideal would
not have codimension two). 
So let 
\[ u = u_0 \mto{e_1} u_1 \mto{e_2} \ldots \mto{e_r} u_r = v  \]
be the path from $u$ to $v$. Its $z_1$ exponents must be non-decreasing. Now
$M_v = \Pi_{i=1}^r y_{e_i} \times \Pi_{e \neq e_i} y_e$ and the $z_1$ 
exponent of the image of the 
first factor is precisely the exponent of $z_1$ in $m_v$. We must then show that the $z_1$ 
exponent of the image of 
the second factor is zero, i.e. $z_1$ does not occur in $\zz^{a_e}$ for any $e \neq e_i$.

If $s \mto{e} t$ is any other edge in the tree, then deleting $e$, note that $t, v$ and $u$ 
will be  in the same connected component. So the path from $s$ to $u$ must pass through $t$. 
But then the $z_1$ exponents along this path must be non-increasing, and then $\zz^{a_e}$ does
not contain $z_1$. 

\medskip
In view of the uniqueness of the morphism, it is immediate that the labelling (\ref{TreLabMv})
is maximal.

\end{proof}

\rem In \cite{Ph}, J.Phan studies the LCM lattice of a monomial 
ideal, or rather he starts from an atomic lattice and studies monomial
ideals with this as their LCM lattice. He shows that for every such 
lattice there is a distiunguished square free monomial ideal, which he
calls a minimal monomial ideal, whihc has this lattice as the LCM lattice.

If one considers the maximal monomial labelling we have on a tree, this is in
fact a minimal monomial ideal on the atomic lattice it generates. Note
however that for a given tree $T$, the objects in $\CM(T)$ will give many
different LCM lattices. In particular Example \ref{TreEksStreng} shows that
on any tree with $n$ vertices 
there is a monomial labelling whose atomic lattice is the
lattice $\hat{L}$ where $\hat{L} \backslash\{ 0,1\}$ is the antichain
of $n$ elements. In fact this monomial labelling is also the minimal monomial 
ideal for this lattice.

An interesting question is : Given two ideals
$I$ and $J$ in $\CM(n,2)$ which have the same LCM lattice, do they have
the same set of trees which give a cellular resolution?
If so it would induce a nice correspondence between families of 
trees and families of atomic lattices.

%If this is so there would be a nice correspondance between certain families
%of trees $\mathcal T$ and certain families of atomic lattices $\mathcal A$. 
%The family $\mathcal A$ would be all atomic lattices coming from ideals
%for which every tree in $\mathcal T$ gives a cellular resolution

%\rem 
%To any square free ordered set $\aa$ of monomials generating an ideal such that the quotient 
%ring is Cohen-Macaulay
%of codimension two, the maximal objects in $\CM(A)$ must by Proposition \ref{TreProEks} 
%and the above Theorem
%\ref{TreTheMax} correspond to $[n]$-labelled trees. We thus get a family of labelled trees. 
%For the sequence of 
%monomials in Example \ref{TreEksAlle} this family is all trees. 
%By Proposition \ref{MaxProLatt} there is thus a boolean
%lattice whose elements are families of $[n]$-labelled trees.

\section{Subdivisions of polygons}

Consider a polygon whose vertices are labelled by $0,1, \ldots, n-1$ and let
$X$ be a subdivision without introducing new vertices. 
I.e. we get the subdivisions by introducing
chords in the polygon. We shall use the notation $[i,j]$, called a {\it string} to 
denote the vertices obtained by starting at $i$ and increasing by 
one each step modulo $n$ until we reach $j$. The length of a string is the number
of vertices in it. We denote by $V$ the set of vertices $\{0,1,\ldots, n-1\}$.

  We want to describe the maximal monomial labellings in $\CM(X)$.
 By Lemma \ref{MaxLemXbet} we may
assume that these are in $\CM_\dagger(X)$. To a monomial labelling on $X$, 
there is associated a family
$\gF$ of subsets of the vertices $V$ which determines the monomial 
labelling up to isomorphism (i.e. permutation of variables).

\begin{lemma} Let $\gF$ be the family of subsets of $V$ associated to
an object in $\CM_\dagger(X)$. Then each element in $\gF$ is a string $[i,j]$.
\end{lemma}

\begin{proof} 
Let $s$ be an element of $\gF$ and suppose $s$ is not a string. Then $s$ is 
a disjoint union of two or more strings, none of which are adjacent.
If there is a chord between any of these strings, then $X$ restricted to the complement
of $s$ is not connected, contradicting Propositin \ref{MaxPropFamKrit}.
Hence there are no chords between strings in $s$, and so $X$ restricted to
$s$ is not connected, contradicting Lemma \ref{MaxLemXbet}.

%If there is no chord from 
%$s_1$ to $s_2$ the restriction $X_{|s_1 \cup s_2}$ cannot be connected, contradicting
%Lemma \ref{MaxLemXbet}. On the other hand if there is a chord from $s_1$ to
%$s_2$, then $X$ restricted to the complement of $s$ will not be connected,
%contradicting Proposition \ref{MaxPropFamKrit}.
\end{proof}

\begin{lemma} \label{PolyLemEndSta} For every vertex $i$ there is a string in $\gF$ 
ending at $i$ and a string in $\gF$ starting at $i$.
\end{lemma}

\begin{proof} 
Since the monomials constitute a minimal generating set for the ideal,
there must be some variable in $m_i$ which is not in $m_{i-1}$.
But then to this variable the associated vertex subset must be
a string starting at $i$. Similarly we get a string ending at $i$ by considering
$m_{i+1}$.
\end{proof}

\begin{lemma} \label{PolyLemCov}
Let $i$ be a vertex which is not the end of a chord. Then in $\gF$ there is a 
string ending at $i-1$  and a string starting at $i+1$ whose union cover $V\backslash \{i\}$.
\end{lemma}

\begin{proof} By Lemma \ref{PolyLemEndSta} there are strings ending at $i-1$ and
starting at $i+1$. Letting $W$ be their union, we know that $X$ restricted to the
complement of $W$ is acyclic. But $i$ is an isolated vertex here, and
so this must be the only vertex in the complement of $W$.
\end{proof}

\begin{lemma} \label{PolyLemSkyv}
Given a string $s$ in $\gF$ starting at $i$, which is not the endpoint of a chord.
Then there is a string in $\gF$ starting at $i+1$ of length greater or equal to that
of $s$. Similarly if $s$  ends at $j$, not the endpoint of a chord, there
is a string ending at $j-1$ of greater or equal length than that of $s$.
\end{lemma}

\begin{proof}
There is a string $t$ starting at $i+1$ which together with a string ending at
$i-1$ covers $V \backslash \{i\}$. It $t$ was contained in $s$, then $s$ and the 
string ending in $i-1$ would cover $V$. Impossible. Hence $t$ is not contained in $s$,
and its length must be at least that of $s$.
\end{proof}

%By Lemma \ref{PolyLemCov} there are strings $[l,i-1]$ and $[i+1,k]$ covering 
%$V\backslash \{i\}$.
%Thus $l$ must be contained in $[i+1,k+1]$. If $k$ is in $[i+1,j]$ then $[i,j]$
%and $[l,i-1]$ together cover the vertex set $V$, impossible. Hence $k$
%is not in $[i,j]$ and so the length of $[i+1,k]$ is greater or equal to that
%of $[i,j]$.
%\end{proof}

Now we are ready to do the case when $X$ is a polygon, i.e. it contains no chords.
%We show that there is only one maximal element when $X$ is a polygon,
%and we give its description.

\begin{theorem} \label{PolyThePoly}
Let $X$ be a polygon and $\gF$ the family of strings associated
to an object in $\CM_\dagger(X)$.

a. If the number of vertices is odd, $n = 2r+1$, then $\gF$ consists of all
strings of length $r$. Hence, up to isomorphism, there is only one object in $\CM_\dagger(X)$
and this is the only maximal monomial labelling of $X$ (up to isomorphism). 

b. If the number of vertices is even, there are no monomial labellings of $X$ giving
ideals of codimension three with Cohen-Macaulay quotient ring, i.e. $\CM(X)$ is 
empty.
\end{theorem}

\rem This result is known but we don't know of a specific reference. 
It is of course closely related to the Buchsbaum-Eisenbud
structure theorem in  \cite[Thm. 2.1]{BE}.
\remfin

\begin{proof}
Let $L$ be the length of the longest string in $\gF$. By Lemma \ref{PolyLemSkyv} all
strings of this length must be in $\gF$. Since two strings cannot cover the vertices,
we must have $2L < n$.
If $n$ is even, equal to $2r$, then $L < r$. But then there
will be two disjoint nonadjacent strings of length $L$, and so $X$ restricted
to the complement of the union of these is disconnected. Hence $n$ cannot be 
even. If $n$ is odd, equal to $2r+1$, a similar argument gives that $L$ cannot 
be less or equal to $r-1$. Hence we must have $L = r$. Again an argument as above
gives that there can not be in addition any string of length less or equal to $r-1$. And so 
all strings have length $r$. 
   It is easy to check that the family of all strings of length $r$ fulfils the criteria
of Proposition \ref{MaxPropFamKrit}, and so it corresponds to an object in $\CM_\dagger(X)$.
Being the only object of $\CM_\dagger(X)$ (up to isomorphism), it must be a maximal monomial
labelling of $X$.
\end{proof}

Now in the rest of this section, we shall consider the case of a polygon with one chord. 
In this case we shall show that there are exactly two maximal
objects in $\CM(X)$. They must be in $\CM_\dagger(X)$ and so the families $\gF_1$ and
$\gF_2$  of subsets
of vertices will consists of strings. Let us describe these families of strings
explicitly. 

\medskip
\noindent {\bf Maximal monomial labellings of $X$.} Let the chord be between 
$0$ and $a$ where $2a \leq n$.
Suppose the number of vertices $n$ is odd, equal to $2r+1$. 
The first family $\gF_1$ is the family in Theorem \ref{PolyThePoly} a. extended by adding
one string. It consists of :

\begin{itemize}
\item[a.] All strings of length $r$. 
\item[b.] The string $[1,a-1]$.
\end{itemize}

\noindent The family $\gF_2$ is given by :

\begin{itemize}
\item[a.] All strings of length $r+1$ containing $[0,a]$.
\item[b.] All strings of length $r$ containing i) $0$ but not $a-1$,
or ii) $a$ but not $1$.
\item[c.] All strings of length $r-1$ disjoint from $[0,a]$.
\item[d.] The string $[1,a-1]$.
\end{itemize}

\noindent In both cases the families consists of $n+1$ subsets of vertices,
and hence the maximal labellings of $X$ are monomials in $n+1$ variables.

\medskip

Suppose the number of vertices $n$ is even, equal to $2r$.
The family $\gF_1$ is given by :

\begin{itemize}
\item[a.] All strings of length $r$ containing $0$.
\item[b.] All strings of length $r-1$ not containing $0$ and $1$.
\item[c.] The string $[1,a-1]$.
\end{itemize}

\noindent The family $\gF_2$ is the mirror image of the family $\gF_1$ :
\begin{itemize}
\item[a.] All strings of length $r$ containing $a$.
\item[b.] All strings of length $r-1$ not containing $a$ and $a-1$.
\item[c.] The string $[1,a-1]$.
\end{itemize}

\noindent Again in both cases the families consists of $n+1$ subsets of vertices,
and hence the maximal monomial labellings of $X$ are monomials in $n+1$ variables.

\begin{theorem} \label{PolyThmKorde} Let $X$ be a polygon with a chord.
Then there are two maximal objects in $\CM(X)$, and their associated families of
subsets of vertices are given by $\gF_1$ and $\gF_2$ above.
\end{theorem}

We shall proceed to prove this theorem through a series of lemmata.
In the lemmata below we let $\gF$ be the family of subsets of vertices arising from an
object of $\CM_\dagger(X)$. 

\begin{lemma}
 \label{PolyLemCompl} There is no string in $\gF$ containing the complement of $(0,a)$.
\end{lemma}

\begin{proof}
Let $i$ be in the complement of $[0,a]$. By Lemma 
\ref{PolyLemCov} there is a string ending in $i-1$ and a string starting in $i+1$, together
covering $V \backslash \{i\}$. At least one of them has length $\geq (n-1)/2$, say the one ending
in $i-1$. We can then by Lemma \ref{PolyLemSkyv} push it backwards till it ends in $a$. 
Then it must start in $1$ or earlier.
If there were a string containing the complement of $(0,a)$, these two strings would cover $V$.
Impossible.
\end{proof}

%Let $i$ be in the complement of $[0,a]$. Then there are strings starting at $i+1$ and
%ending at $i-1$ whose
%union cover $V\backslash \{i\}$, and so one of them has length $\geq n-1/2$, 
%say the one starting at
%$i+1$. We may then by Lemma \ref{PolyLemSkyv} slide this string until we get
%a string starting at $0$ of length $\geq n-1/2$. 
%This string must then contain $[0,a]$.
%\end{proof}

\subsection{The case of $n$ even}

We now assume that the polygon has an even number of vertices $2r$.

\begin{lemma} \label{PolyLemJamni}
Let $i$ be in $(0,a)$. Let $s_1$ be a string ending in 
$i-1$ and $s_2$ a string starting in $i+1$. Then $s_1$ has length $r$ and 
$s_2$ length $r-1$, or conversely.
\end{lemma}

\begin{proof} Suppose $s_1$ has length $\geq r+1$. By Lemma \ref{PolyLemSkyv}
we can then step by step push it back
to a string ending at $0$ which also has length $\geq r+1$. 
It does not begin in $[0,a]$ by Lemma \ref{PolyLemCompl},
so we can also push
it forward to a string starting at $0$ and having length $\geq r+1$. 
Together with the one ending in $0$, these two would cover the vertices
of $X$. Impossible. 
%If $s_1$ does begin in $[0,a]$, by Lemma \ref{PolyLemOA} there is 
%a string covering $[0,a]$ and together these two strings would cover
%$X$. Impossible.

  Thus $s_1$ has length $\leq r$, and similarly $s_2$ has length $\leq r$.
If $s_2$ has length $\leq r-2$, the complement of $s_1 \cup s_2$ consists
of $\{i\}$ together with another disjoint component. Impossible. Thus the strings
have length $r$ or $r-1$. If both had length $r-1$, the complement of 
$s_1 \cup s_2$ will again consist of at least two components. If both had length
$r$, we could push the end of $s_2$ one step back to get $s_2^\prime$ such 
that $s_1 \cup s_2^\prime$ covers $V$. Impossible. Thus the lengths are $r$
and $r-1$.
\end{proof}

Thus either there is a string of length $r$ containing $0$ but not $a$, or
there is a string of length $r$ containing $a$ but not $0$. These cases turn
out to be mutually exclusive and seperates the treatment into two cases. These
are symmetric and we shall consider the first case.

\begin{lemma} \label{PoljLemOnA} Suppose a string in $\gF$ containing $0$ but not $a$, 
has length $r$. 

a. Then all strings in $\gF$ containing $0$ have length $r$, and all strings
of length $r$ containing $0$ are in $\gF$.

b. All strings in $\gF$ which do not contain $0$ or $1$ have length $r-1$, and
all strings of length $r-1$ not containing $0$ and $1$ are in $\gF$.

\end{lemma}

\begin{proof} Let $s$ be the string of length $r$ in $\gF$ containing $0$ but not 
$a$. We can push it successively forward until we get a string starting in $0$.
We can also start with $s$ and push it successively backwards to a string ending 
in $0$. If the  length jumps up at some stage, these two strings will cover $V$. 
Impossible. Thus all these strings have length $r$.

Now let $s$ be a string starting in $2$. By Lemma \ref{PolyLemJamni}, $s$
has length $r-1$. We can successively push it forward, one step a time, until it
starts in $a$. If it now had length $\geq r$, then since there is a string of length
$r$ ending in $a-1$, these would cover $X$. Impossible. 
Thus all of these strings have length $r-1$. 

Let now $s$ be a string in $\gF$ ending in $-1$. We can successively push it
backwards, getting $s^\prime$ ending in $-r$. Since there is a string
$[-r+1,0]$, $s^\prime$ must start  at $2$ or later, in order to avoid two
strings covering $V$. Thus all such strings have length $\leq r-1$.
Now there is a string $t$ starting at $2$ of length $r-1$. If $s$ (ending in $-1$)
had length $\leq r-2$, the complement of $s \cup t$ would be disconnected.
Thus $s$ has length $r-1$, and we get part b.
\end{proof}

\begin{lemma} In the situation of Lemma \ref{PoljLemOnA}, 
the family $\gF$ has one string starting in $1$ and
this is the string $[1,a-1]$.
\end{lemma}

\begin{proof}
Let $s$ begin in $1$. Then $s$ has length $\leq r-1$ since
else we could push it forward one step and get a string of length $r$ starting in $2$.
Impossible. 
If $s$ ends in a point $\geq a$, let $t$ end in $-1$, note that $t$
has length $r-1$. Then $s \cup t$ has a complement which is disconnected.
Hence $s$ ends in a point in $(0,a)$, According to Lemma \ref{PolyLemJamni}
its endpoint must be $a-1$.
\end{proof}

\begin{proof}[Proof of Theorem \ref{PolyThmKorde} [Even number of vertices]]
The Lemmata above show that a maximal family in $\CM_\dagger(X)$ must be either
$\gF_1$ or $\gF_2$. It is an easy matter to verify that they also satisfy
the criteria of Proposition \ref{MaxPropFamKrit}.
\end{proof}

\subsection{The case of $n$ odd} \label{SecPolyodd}

We now assume that the number of vertices of the polygon is odd equal to 
$2r+1$. We let $\gF$ be a maximal CM family of subsets of vertices.

\begin{lemma} \label{PolyLemOddi}
a. Let $i$ be in $(0,a)$. Let $s_1$ end in $i-1$ and $s_2$ begin in $i+1$.
Then both $s_1$ and $s_2$ have length $r$. 

b. All strings in $\gF$ have length $\leq r+1$. 
\end{lemma}

\begin{proof} Part a. goes as Lemma \ref{PolyLemJamni}.
To show part b., let $s$ be a string of length $\geq r+2$. It does not both 
start and end in $[0,a]$ by Lemma \ref{PolyLemCompl}. 
Suppose it ends outside of $[0,a]$. Then we can push
it backward to a string containing $0$, and then forward to a string
starting in $0$, of length $\geq r+2$. But by part a. there is a string ending in $0$ of
length $r$ and these two would cover $V$. Impossible.
\end{proof}

\begin{lemma}
All strings in $\gF$ containing $[0,a]$ have the same length, and this length
is either $r$ or $r+1$. Either i) all strings of length $r+1$ containing
$[0,a]$ are in $\gF$,  or ii) all strings of length $r$ containing $[0,a]$ 
are in $\gF$, in which case $a < r$, or iii) there are no strings containing $[0,a]$, 
in which case $a = r$.  
\end{lemma}

\begin{proof} There is a string $s$ ending in $0$, it has length $r$ by
Lemma \ref{PolyLemOddi}. By Lemma \ref{PolyLemCompl} it does not start in 
$[0,a]$. We may then push it forward to a string $t$ starting in $0$, having
length $\leq r+1$. Suppose the length is $r+1$.
We can freely push it
back and forth over $[0,a]$ and its length will 
not decrease by Lemma \ref{PolyLemSkyv}, and not increase by Lemma 
\ref{PolyLemOddi}, so
all such strings are in $\gF$. Suppose there in this case also were a string $s$
in $\gF$ of length $r$ containing $[0,a]$. Suppose the distance from its
start point $i+1$ to $0$ is less or equal to the distance from its end point to $a$. 
There is a string $u$ ending in $i-1$ which together with $s$ covers
$V\backslash \{ i\}$. The length of $u$ cannot be $r+1$ since we then could
push $u$ one step forward and this string together with $s$ would cover $V$.
Thus $u$ has length $r$. But then the complement of $u$ would be a string
$u^\prime$ containing $[0,a]$ of length $r+1$ and thus be in $\gF$. Thus
$s$ and $u^\prime$ together would cover $V$. Impossible.

\medskip
Suppose now there is a string $t$ of length $\leq r-1$ containing $[0,a]$, 
and let $t$ end in $i-1$. There is
a string $s$ starting in $i+1$. Since $s \cup t$ must be $V \backslash \{ i \}$,
$s$ must have length $\geq r+1$. We can then successively push $s$ forward
till it starts in $r+1$. Its length must by Lemma \ref{PolyLemOddi}.b be $r+1$,
and so it ends in $0$. But this is impossible by part a. of the same lemma.
Thus the length of $t$ is $r$ or $r+1$.

\medskip
Suppose now the string $t$ in the beginning of the proof had length $r$. (Then $a < r$.)
We can push it back
and forth over $[0,a]$. All these strings
are in $\gF$, and the length does not increase since by the first part of the proof, strings
covering $[0,a]$ cannot both have length $r$ and $r+1$.

\end{proof}

The argument now splits into two cases corresponding to whether 
$\gF$ contains strings of length $r+1$ or not. These will give the
two maximal families $\gF_1$ and $\gF_2$.

\begin{lemma} a. Suppose $\gF$ has strings of length $r+1$. 
Then all strings in $\gF$ disjoint from $[0,a]$ have length $r-1$
and all such strings of length $r-1$ are in $\gF$.

b. Suppose $\gF$ has no strings of length $r+1$.  Then 
all strings in $\gF$ disjoint from $[0,a]$ have length $r$
and all such strings of length $r$ are in $\gF$

\end{lemma}

\begin{proof} 
a. Suppose $s$ is disjoint from $[0,a]$. If its length is $r$,
its complement has length $r+1$ and covers $[0,a]$. Thus it is also in $\gF$.
Impossible since the vertices of $X$ cannot be covered by two sets in $\gF$.
The length of $s$ cannot be $\geq r+1$ by a similar argument.
Suppose $s$ has length $\leq r-2$. 
So suppose it ends in $i-1$ distinct from $-1$. If $t$ begins in 
$i+1$ it must together with $s$ cover $V\backslash \{ i\}$ and thus have 
length $\geq r+2$. Impossible. Similarly it cannot start in a vertex distinct from
$a+1$. So it would have to start in $a+1$ and end in $-1$. But that
is not possible if the length is $\leq r-2$.
Thus $s$ is of length $r-1$.

Now there is a string starting in $a+1$. Since it cannot cover $V$ together
with a string containing $[0,a]$, its endpoint must be before $0$. 
Thus it has length $r-1$. The same holds also for strings ending in $-1$.

We now push the string starting in $a+1$ successively forward until we get
a string $t$ starting in $-r$. The endpoint of $t$ is not $\geq 0$
since it would then cover $V$ with the string $[0,r]$. Thus $t$ is in 
the complement of $[0,a]$ and has length $r-1$, and ends in $-2$. 

The proof of part b. is analogous.

\end{proof}

\begin{lemma} a. Suppose $\gF$ has strings of length $r+1$. 
Then $[1,a-1]$ is in $\gF$ and is the only string starting in $1$
or ending in $a-1$.

b. Suppose $\gF$ has no strings of length $r+1$. Then $[1,r]$ is in 
$\gF$, and the only other possible string in $\gF$ starting in $1$ is 
$[1,a-1]$.
Correspondingly $[a-r,a-1]$ is in $\gF$ and the only other possible
string ending in $a-1$ is $[1,a-1]$.
\end{lemma}

\begin{proof} a. Let $s$ start in $1$. Then it is of length $\leq r$ since
otherwise we could push it forward to a string starting in $2$ of length
$\geq r+1$. Impossible by Lemma \ref{PolyLemOddi}.
If it ends in a point $\geq a$, let $t$ be a string of length $r-1$ ending
in $-1$.
Then the complement of $s \cup t$ would be disconnected. Impossible.
Thus $s$ ends in a point in $(0,a)$. According to Lemma \ref{PolyLemOddi}
it must then end in $a-1$. The argument concerning the string ending in 
$a-1$ is analogous.

b. Let $s$ start in $1$. As above its length must be $\leq r$. Suppose the length
is $< r$. If it ends in $(0,a)$ its endpoint must be $a-1$ according to Lemma 
\ref{PolyLemOddi}.
That $s$ ends in a point $\geq a$ is as above impossible.

The string starting in $2$ has length $r$. Pushing it backward one step,
its length cannot increase. Thus we get a string of length $r$ starting
in $1$.
\end{proof}

\begin{proof}[Proof of Theorem \ref{PolyThmKorde} [Odd number of vertices]]
The lemmata above show that if $\gF$ has a string of length $r+1$, it must be the
family $\gF_2$. Also if $\gF$ does not contain strings of length $r+1$, it must 
contain all strings of length $r$ and the only possible extra string being $[1,a-1]$.
But adding this string gives a family fulfilling the criteria of Proposition \ref{MaxPropFamKrit}.
Hence if $\gF$ is maximal it must be $\gF_1$.
\end{proof}

We formulate the following conjecture concerning maximal CM monomial 
labellings of subdivisions of polygons.

\begin{conjecture} \label{PolyConjVar}
Given an $n$-gon with a subdivision of $k$ additional edges (but
no additional vertices). Then any maximal CM monomial labelling
has $n+k$ variables.
\end{conjecture}

\eks A cellular resolution of the monomials of degree two in three
variables is given by removing one of the interior edges in Figure 3.1.

\medskip
\trekantA
\medskip

We may polarise and get Figure 3.2.
However this is still not a maximal monomial labelling. Removing
the edge between $xy$ and $yz$ in the first diagram, there is
another way of ``polarising'' this, Figure 3.3. (The point here is that
the indices are always $1$ and $2$.)

\medskip
\trekantB
\medskip

Combining these two last diagrams we get a maximal monomial labelling
with eight variables in Figure 3.4.
Note that there are two possible maps from this monomial labelling
to the one in Figure 3.1, factoring through the second and
third figure respectively.

\section{Cohen-Macaulay monomial labellings of higher dimensional polytopes}

This section investigates CM cellular resolutions supported on
polytopes. We give constructions of CM monomial labellings
on some classes of selfdual polytopes where the class contains polytopes of 
arbitrary dimensions. These labellings are shown to be maximal.
We also conjecture that any polytope supporting a CM cellular resolution
must be selfdual.

In the end we consider an example of a three dimensional selfdual polytope
and construct a maximal CM monomial labelling of it. Along the way
we give several examples of how the labelled polytopes give cellular
resolutions of Stanley-Reisner rings of simplicial polytopes.

In this section our cell complex $X$ will be a convex polytope.
%But the results easily generalizes to a regular cell complex $X$
%which is topologically a ball and which has a single cell of 
%top dimension.

\subsection{Necessary conditions on CM monomial labellings}

\begin{lemma} Let $P$ be a polytope supporting a CM cellular
resolution. Then its $f$-vector is symmetric, i.e. if $f_i$ is the number
of $i$-dimensional cells, then $f_i = f_{\dim P-1-i}$.
\end{lemma}

\begin{proof} By polarising we may assume the monomial labelling of $P$ is square free.
Since $P$ is a polytope, the corresponding cellular resolution has
type $1$. But a simplicial complex which is Cohen-Macaulay of type $1$
is Gorenstein. Hence the resolution is self-dual and so 
$f_i = f_{\dim P-1-i}$.
\end{proof}

\begin{conjecture} If $P$ is a polytope supporting a CM cellular 
resolution, it is a selfdual polytope.
\end{conjecture}

\eks 
A three-dimensional polytope which is a bipyramid cannot support
a CM cellular resolution.
If the base is an $n$-gon, where $n \geq 3$, it has
$n+2$ vertices and $2n$ faces. But then $2n \neq n+2$.
\eksfin

\subsection{CM monomial labellings of pyramids}

Given a polytope $X$, let $PX$ be the pyramid over $X$, i.e. the convex
hull of $X$ and a point $t$ outside the linear space where $X$ lives.
The polytope $X$ is selfdual iff $PX$ is.

\begin{theorem} Given a CM monomial labelling $(\aa_1, \ldots, \aa_v)$
of $X$. Let $y$ be a variable.

a. Then $(\aa_1, \ldots, \aa_v, y)$ is a CM monomial labelling of $PX$.

b. The labelling of $X$ is maximal iff the associated labelling of 
$PX$ is maximal.

c. Every maximal CM monomial labelling of $PX$ is of this form.

\end{theorem}

\begin{proof}
a. If $C$ is a cellular resolution over $S$ of the quotient ring
$S/(a_1, \ldots,a_v)$,
then the cone of $C\te_k k[y](-1) \mto{\cdot y} C \te_k k[y]$
is a resolution of the quotient ring of $S \te_k k[y]$
given by the labelling of $PX$.

\medskip
b. Let the monomial labelling $(\aa_1, \ldots, \aa_v)$ correspond
to a family $\gF$ of subsets of $V$. We need to show that $\gF$ is maximal
iff $\gF \cup \{ \{ t\} \}$ is maximal.

If $\gF$ is maximal, then if we could add a subset $\tilde{T}$ of 
$V \cup \{  t \}$ to $\gF \cup \{ \{ t \} \}$ and still have the 
conditions of Proposition \ref{MaxPropFamKrit} holding, 
then we could add $\tilde{T} \cap V$ to 
$\gF$, Proposition \ref{MaxPropFamKrit} still holding.
Since $\gF$ is maximal, $\tilde{T} \cap V$ would have to be a disjoint union
of sets in $\gF$. But then the same would hold for $\tilde {T}$.
Therefore $\gF \cup \{ \{ t \} \}$ is maximal.
%contradicting the maximality of this family. 

Conversely, if $\gF \cup \{ \{ t\} \} $ is maximal, then if we add
a subset $T$ to $\gF$ and still have the criteria of Proposition \ref{MaxPropFamKrit}
holding, then we could add this to $\gF \cup \{ \{ t\} \} $ with the criteria still
valid. Hence $T$ is a disjoint union of sets in $\gF \cup \{ \{ t \} \}$ and so it must
be a disjoint union of sets in  $\gF$, and so $\gF$ is maximal.

%contradicting
%the maximality of this family.

\medskip
c. Suppose the maximal family $\gF$ on $PX$ contains $S \cup \{t\}$ where $S$ is nonempty.
Let $W = V \backslash S$ and $Y = X_{| W}$. 
Let $\gF^\prime = \{ T \cap W\, |\, T \in \gF \}$. Then the family $\gF^\prime$ has the
property that the complement of any union of elements in $\gF^\prime$ gives
an acyclic restriction.
Hence the family $\gF^\prime$ determines a cellular resolution of the ideal generated
by $\{ \aa_i \}_{i \in W}$. Since $\dim Y < \dim X$, the quotient ring by this ideal 
will have codimension $\leq \dim Y + 1 \leq \dim X$. Thus $Y$ can be covered by
$\leq \dim X$ elements in $\gF^\prime$. But then $PX$ can be covered by $S \cup \{ t \}$
together with these, a total of $\leq \dim X + 1$ elements of $\gF$. But this is impossible
for a CM labelling.
\end{proof}

%Now let $\gG$ be a maximal family of subsets of the vertices of $PX$.
%Let $S \cup \{ t \}$ be in $\gG$. We need to show that $S$ is empty. Suppose
%to the contrary that $s \in S$.
%By Proposition \ref{MaxPropT} there is a family $\gG^\prime$ of
%$\dim X + 2$ subsets covering $PX$, with only $S \cup \{ t \}$ containing
%$s$. Let $\gF^\prime$ be the family $\{ T \cap V \,|\, T \in \gG \}$ 
%considered on $X$. If  all complements of unions of elements in $\gF$
%are acyclic, then the corresponding monomial labellings must give
%a cellular resolution of some monomial ideal. By the Auslander-Buchsbaum
%theorem, the quotient ring of this monomial ideal has codimension 
%$\leq \dim X + 1$. Hence $X$ can be covered by $\dim X + 1$ of the sets 
%in $\gF$. But note that the only element in $\gF$ containing $s$ in
%is $S$. Thus $S$ must be in this covering. But then $PX$ can  
%also be covered by $\dim X + 1$ elements in $S$. Impossible.
%\end{proof}

\eks A three-dimensional bipyramid is a simplicial polytope.
In order for its Stanley-Reisner ideal to 
have a cellular resolution supported on a three
dimensional polytope, the base of the bipyramid must be a pentagon
(the number of vertices must be four more than the dimension).
In this case a minimal cellular resolution is given by the pyramid
over the pentagon labelled as follows. (A label $ij$ is short for $x_ix_j$.)

\medskip
\pentpyramid
\medskip
\eksfin

\subsection{CM monomial labellings of elongated pyramids}
The elongated pyramid $EPX$ over $X$ is the union of $X \times [0,1]$
and the pyramid over $X$, with the base of the pyramid identified
with $X \times {1}$. The polytope $EPX$ is selfdual iff $X$ is.
If $V$ is the vertex set of $X$ let $EPV$ be the vertex set of 
$EPX$. It consists of $V \times \{0,1\}$ and the vertex $t$ of the 
pyramid. Given a family $\gF$ of subsets of $V$, we get a family $EP\gF$
of subsets of $EPV$ consisting of (letting $S$ vary over $\gF$)
i) $S \times \{1\} \cup \{ t \}$, ii) $S \times \{ 0,1 \}$, 
iii) $V \times \{ 0 \}$.

\begin{theorem} Let $\gF$ be in $\CM_*(X)$. 

a. The family $EP\gF$ is in $\CM_*(EPX)$.

b. If $\gF$ is maximal, then $EP\gF$ is maximal.
\end{theorem}

\begin{proof}
a. The family $\gF$ fulfils the conditions of Proposition \ref{MaxPropFamKrit}. We must show
that $EP\gF$ fulfils the same conditions.

\medskip 

\noindent Conditon 3. Let $F \sus G$ be two distinct faces of $EPX$. Suppose $G$ is 
contained
in the pyramid over $X \times \{1\}$ with vetex $t$. If $G$ is contained in the base, 
condition
3. is clear. If $G$ is the pyramid over a face of $X$, then if $F$ is strictly contained
in this base face, condition 3. is clear. If $F$ is the base face, there is a set $S$
in $\gF$ disjoint from $F$ (by considering the inclusion of faces $F \sus X$), and then
condition 3. holds by considering $S \cup {t}$. 
If $G$ is contained in $X \times \{0\}$, condition 3. clearly also holds. Suppose now
that $G$ is $A \times [0,1]$ where $A$ is a face of $X$. If $F$ is contained in
$B \times [0,1]$ where $B$ is a face strictly in $A$, condition 3. is clear. Otherwise
$F$ is either $A \times \{0\}$ or $A \times \{1\}$ in which case condition 3. is also clear.

%\noindent Condition 3. Let $F \sus G$ be two distinct faces of $EPX$. If they have different
%projections on $V$, condition 3. is easy. If $G$ is $T \times \{0,1\}$
%and $F$ is $T \times \{ 0 \}$, there is a set $S \times \{1\} \cup \{t \}$
%fulfilling condition 3, and if $F$ is
%$T \times \{ 1 \}$, the set $V \times \{ 0 \}$ fulfils condition
%3. If $G$ is $T \times \{ 1 \} \cup \{ t \}$ and $F$ is $T \times \{ 1 \}$
%we know that there is an $S$ in $\gF$ such that $S \cap T$ is empty
%(else condition 3 would not hold for the inclusion $T \sus X$).
%Then $S \cup \{ t\}$ fulfils condition 3.

\medskip

Condition 1. Let $\tilde{S}_1 \cup \cdots \cup \tilde{S}_m$ be a covering of $EPV$
by sets of $EP\gF$. If none of these are $V\times \{ 0 \}$, then since
they cover $V \times \{ 0 \}$, a selection of, say $r$ of  these must be of the
form $S \times \{0,1 \}$ where these $S$'s cover $V$. But then
$r \geq \dim X + 1$. Since something also must contain $t$, we
get $m \geq \dim X + 2$.

If one of the above is $V \times \{0\}$, the rest will cover $V \times \{1\}$.
We need at least $\dim X + 1$ such and so $m \geq \dim X + 2$.

\medskip

Condition 2. Let $\tilde{S}_1 \cup \cdots \cup \tilde{S}_m$ be a union of subsets of $EP\gF$.
We will show that $X$ restricted to its complement is acyclic.
Suppose first that $V \times \{ 0 \}$ is one of the sets in the union.
If none of the $\tilde{S}_i$ contains $t$, the complement is contractible.
If some $\tilde{S}_i$ contains $t$, the acyclicity of the complement
follows by the fact that this is true for $X$.

Suppose then that $V \times \{0 \}$ is not one of the sets in the union.
Then the complement will consist of the union of $A \times \{1 \}$ and
$B \times \{ 0 \}$, where $A \sus B$, and possibly $t$, 
where $X$ restricted to $A$ or $B$ are both acyclic. 
Then  $EPX$ restricted to this complement is also acyclic.

\medskip

b. We now want to show that if $\gF$ is a maximal family,
then so is $EP\gF$. Let $T \sus EPV$ be a set of vertices. Suppose 
$EP\gF \cup \{T\}$ fulfils the criteria of Proposition \ref{MaxPropFamKrit}.

Let $T^0 \times \{ 0 \} = T \cap (V \times \{ 0 \} )$. Suppose 
$T^0$ is not in $\gF$. The reason is either that i) it is empty or
ii) $X$ restricted to
the complement $C$ of the union of $T^0$ and various other elements of 
$\gF$ is not acyclic, or iii) $T^0$ covers $V$ together with $\leq \dim X - 1$
elements in $\gF$, or iv) it is a disjoint union of sets in $\gF$.

In case ii) $C \times \{ 0 \}$ will also be a complement of a union of
elements of $EP\gF$. Impossible. 
In case iii), let $T^0 \cup S_1 \cup \cdots \cup S_r $ cover $V$ where 
$r \leq \dim X - 1$.
Now $T$, the subsets $S_i \times \{0, 1 \}$, and $S_1 \times \{1\} \cup \{ t \}$ 
comprise a total of $\dim X + 1$ sets. That is one short of possibly covering $EPV$.
% Since $\cup_{i=1}^r S_i$ does not cover $V$,
%$V \backslash \cup_{i=1}^r S_i$ is not empty. 
If $V \backslash T^0$
is not empty, the complement of $\cup_{i=1}^r (S_i \times \{ 1 \}
\cup \{t\}) \cup T$ will be disconnected:
Its restriction to level 0 and level 1 are non-empty and disjoint, since $T^0$ and the 
$S_i$ cover $V$.
%  On level $1$ it will consist of 
%$V \backslash \cup_1^r S_i$ and on level $0$ of $V \backslash T^0$.
Impossible.
So we can conclude that either $T^0$ is a disjoint union of sets in  $\gF$ or it is empty, or it is $V$.

\medskip

Now let $T^1 \times \{ 1 \} = T \cap (V \times \{ 1 \})$. Suppose $T^1$ 
is not in $\gF$. Then the reason is either that i) it is empty or ii)
$X$ restricted to
the complement $C^\prime$ of the union of $T^1$ and various other elements of 
$\gF$ is not acyclic, or iii) $T^1$ covers $V$ together with $\leq \dim X - 1$
elements in $\gF$, or iv) it is a disjoint union of sets in $\gF$. 

In case ii) $C^\prime \times \{ 1 \}$ will also be the complement of the 
union of elements in $EP\gF$, since $V \times \{ 0 \}$ is in $EP\gF$.
Impossible. In case iii) $EPV$ will be covered by $\dim X + 1$ elements since
$V \times \{ 0 \}$ is in $EP\gF$. Impossible.
So we can conclude that $T^1$ is a disjoint union of sets in $\gF$ or it is empty.

\medskip 
Now suppose  $T^0$ is $V$. If $T^1$ is empty, then since $T$ is not $V \times \{ 0 \}$
(since $T$ is not in $EP\gF$), $T$ must be $V \times \{ 0 \} \cup \{ t \}$. 
But then $EPX$ restricted to the complement of $T$ is a sphere and so not
acyclic. If $T^1$ is a disjoint union of sets in  $\gF$, $EPV$ can be covered 
by $\dim X + 1$ elements according to Proposition \ref{MaxPropT}. Impossible.

Now suppose $T^0$ is empty. If $T^1$ is empty then $T$ must be $\{t\}$. Impossible
since the complement of $T \cup \{V \times \{0\}\}$ will be a sphere and so have homology.
If $T^1$ is a disjoint 
union of two or more sets in $\gF$, then by the following
Lemma \ref{TogS} there 
are sets $S_1, S_2, \ldots, S_{\dim X - 1}$ in $\gF$ such that $T^1$ and these sets cover $V$.
But then $EPV$ will be covered by $\dim X + 1$ subsets of $EP\gF$ if $\dim X \geq 2$. 
Impossible. If $\dim X = 1$, then either $T$ contains $t$ and $EPV$ is covered by
two subsets of $EP\gF$, or $T$ does not contain $t$ and $EPX$ restricted to the complement
of $T$ is disconnected, also impossible. 

%Now suppose $T^0$ is empty. If $T^1$ is empty then $T$ must be $\{t\}$. Impossible
%since the complement of $T \cup \{V \times \{0\}\}$ will be a sphere and so have homology.
%If $T^1$ is a disjoint union of sets in $\gF$, 
%then if $T$ is not already a disjoint union of sets in $EP\gF$ it must be $S \times \{1\}$ for 
%some disjoint union $S$ of sets in $\gF$. 
%Now by Proposition \ref{MaxPropT}, $V$ can be covered by $S$ and $\dim X$ elements
%$S_i$ of $\gF$. The $S_i$'s are not enough to cover $V$. Hence the complement
%of the unions of $S \times \{1 \}$ and the $S_i \times \{ 0,1 \}$ would be disconnected.
%Impossible.

\medskip
Hence $T^0$ must be a disjoint union of sets in $\gF$. Proposition \ref{MaxPropT} guarantees that 
$T^0$ together with $d = \dim X$ elements
$S_1, \ldots, S_d$ in $\gF$ covers $X$. If $T^1$ is empty,
the complement of $\cup_1^d
(S_i \times \{ 1 \} \cup \{ t \}) \cup T$ will be disconnected. Impossible.
Therefore both $T^0$ and $T^1$ are disjoint unions of sets in $\gF$.

If $T^0 = T^1$ and $T$ is $T^0 \times \{ 0, 1 \} \cup \{ t \}$, then
according to Proposition \ref{MaxPropT}, we can cover $EPV$ with $\dim X + 1 $ elements.
Impossible.

Suppose $T^0\backslash T^1$ is nonempty. Let $v$ be in the difference set.
By Proposition \ref{MaxPropT} there exists a covering of $V$ consisting of 
$T^0$ and $S_1, \ldots, S_r$ where $T^0$ is the only set containing $v$.
Then the complement of $T \cup \cup_{i=1}^r(S_i \times \{1\} \cup \{t \})$ 
is disconnected. Impossible.

Hence $T^0 \sus T^1$. Now by Proposition \ref{MaxPropT} there exists $d = \dim X$
elements $S_1, \ldots, S_d$ of $\gF$ which together with $T^0$ cover $V$.
If $T$ contains $t$ then $T$ and the $S_i \times \{ 0,1 \}$, a total of $\dim X + 1$
sets, would cover $EPX$. Impossible. Hence $t$ is not in $T$. 
If $T^1 \backslash T^0$ is nonempty. Let $v$ be in the difference set. 
Again by Proposition \ref{MaxPropT} 
there exists a covering of $V$ consisting of $T^1$ and
$S_1, \ldots, S_r$ where $v$ is only in $T^1$. Then the complement
of $T \cup \cup_{i=1}^r (S_i \times \{ 0, 1 \})$ is disconnected. Impossible.

%Suppose then $T^0 \neq T^1$. Let $S$ be another element in $\gF$. Let 
%$X^i$ be $X$ restricted to the complement of $T^i \cup S$. Then $X^i$
%is acyclic. But since $T^0 \cap T^1$ is not in $\gF$, by Lemma BB,
%$X^1 \cup X^2$ will not be acyclic. Now consider the complement 
%$V^0 \times \{ 0 \} \cup V^1 \times \{ 1 \}$ of 
%$T \cup (S \times \{ 1 \} \cup \{ t\})$. Let $\tilde{X}$ be 
%$EPX$ restricted to this, and let $\tilde{X}^i$ be $EPX$ restricted to
%$V^i \times \{ i \} \cup (V^0 \cap V^1) \times \{ 0,1 \}$. 
%Note that $\tilde{X}^i$ is homotopic to $X^i$ and also $\tilde{X}^1 \cap 
%\tilde{X}^2$ is homotopic to $(X^1 \cap X^2) \times [0,1]$.
%By the Mayer-Vietoris sequence
%\[ H_i(\tilde{X}^1 \cap \tilde{X}^2) \pil H_i(\tilde{X}^1) \oplus
%H_i(\tilde{X}^2) \pil H_i(\tilde{X}) \]
%we get that $\tilde{X}$ is not acyclic. Impossible.
%\medskip
In conclusion $T^0 = T^1$ is a disjoint union of sets in  $\gF$, and  $T$ is $T^0 \times \{ 0,1 \}$.
Thus $T$ is already a disjoint union of sets in  $EP\gF$, and so this family is maximal.
\end{proof}

\begin{lemma} \label{TogS} Let $\gF$ be a maximal family of $CM_*(X)$ where $X$ is a polytope.
Let $T_1$ and $T_2$ be two disjoint sets in $\gF$. Then there are $S_1, S_2, \ldots, 
S_{\dim X -1}$ in $\gF$ such that the union of the two $T$'s and the $S$'s cover the
vertices of $X$. 
\end{lemma}

\begin{proof}
The restriciton of $X$ to the complement of $T_1 \cup T_2$ is acyclic. By Alexander
duality, \cite[Thm. 3.44]{Ha}, the restriction of $X$ to $T_1 \cup T_2$ is also
acyclic and hence connected. So there is an edge $e$ connecting points $t_1 \in T_1$ and
$t_2 \in T_2$.  For each face containing $e$ there is a set $S$ in $\gF$ 
disjoint from $\{t_1, t_2\}$ which includes a vertex of the face. 
The restriction of $X$ to the 
complement $W$ of the union of the $S$'s will be acyclic. In $X_{|W}$ the edge $e$ is
a maximal face. Hence $X_{|W} \backslash \{e\}$ is disconnected, since otherwise
$X_{|W}$ would have nonvanishing $\tH^1$-cohomology. Let $U_i$ be the vertices in the 
connected component of $t_i$ in $X_{|W}\backslash \{e\}$. Then $T_1 \supseteq U_1$ since
otherwise $X$ restricted to $W\backslash T_1$ would be disconnected and this cannot be
so since $W\backslash T_1$ is the complement of a union of sets in $\gF$. Similarly
$T_2 \supseteq U_2$, and so the two $T$'s and the $S$'s cover the vertices of $X$. As
in Proposition 1.13 we may conclude that there are $\dim X - 1$ of the $S$'s that
together with the two $T$'s cover the vertices. 
\end{proof}

\eks The Stanley-Reisner ring of an octahedron with a stellar 
subdivision of one face has cellular resolution given by the elongated
pyramid over a triangle.

\medskip
\elongpyramid
\medskip

\eksfin

\subsection{ CM labellings of a three-dimensional polytope}

Another family of selfdual polytopes of dimension three has plane
diagrams given as follows. Given a $2n$-gon labelled modulo $2n$ by
vertices $0,1, \ldots, 2n-1$. Add a vertex $c$ at the centre and edges
from the centre to each oddly labelled vertex.  Also add edges from
$2i$ to $2i+2$ on the outside. This is the planar graph corresponding
to a selfdual polytope with $2n+1$ vertices. When $n=4$ this may be
displayed as follows.

\medskip
\selvdualpol
\medskip

This has a maximal Cohen-Macaulay labelling given as follows.
\begin{proposition} Given the $3$-polytope $P$ above with $n=4$.
The family of subsets $\{0,1,2\} , \{2,3,4\}, \{4,5,6\}, \{ 6,7,0\}$,
$\{c,1,3\}, \{ c,3,5\} , \{ c,5,7 \} , \{ c,7,1\}$, and
$\{1,2,3\}, \{ 3,4,5\}$
is a maximal family in $\CM_*(X)$.
\end{proposition}

\begin{proof} It is a tedious but straightforward task to show that this
family fulfils the criteria of Proposition \ref{MaxPropFamKrit}.
To show that it is maximal we must also show that this family cannot be
extended, i.e. we cannot add another subset $S$ of vertices, or refine a subset, and still
have all the criteria 1, 2, and 3 of Proposition \ref{MaxPropFamKrit}. 
This is laborious but straightforward.
\end{proof}

Now this $3$-polytope $P$ gives a cellular resolution of the Stanley-Reisner
ideal of various simplicial polytopes where the number of vertices is 
four more than the dimension. We give examples of such simplicial 
polytopes of dimension two and three. 

\eks The hexagon has cellular resolution given by $P$ when labelled
as follows. (A label $ij$ denotes $x_ix_j$.)

\medskip
\firkanthex
\medskip

\eksfin

\eks Consider the bipyramid over the triangle. 
\medskip
\bipullA
\medskip

We can take stellar subdivisions of various pairs of faces. The cellular
resolution of its Stanley-Reisner ring 
is then given by various labellings of $P$. With stellar
subdivision of faces $124$ and $235$ we have Figure 4.6.
With stellar subdivision of faces $124$ and $234$ we have Figure 4.7,
and with stellar subdivision of faces $124$ and $125$ we have Figure 4.8.

\medskip
\bipullB

\end{document}